\documentclass[12pt]{article}
\usepackage[utf8]{inputenc}

\usepackage{color}

\usepackage{epsfig}
\usepackage{cite}
\usepackage{latexsym,amsmath}
\usepackage{amsfonts}
\usepackage{float}
\usepackage{pict2e}
\usepackage{tikz}
\usepackage{caption}
\usepackage{amsfonts}
\usepackage{amssymb}
\usepackage{mathrsfs}
\usepackage{amsmath}
\usepackage{enumerate}
\usepackage{graphicx}
\usepackage{subfig}
\usepackage{MnSymbol}
\usepackage{mathtools}
\begin{document}
	
	\bibliographystyle{plain}
		
		\pagestyle{myheadings}
		\thispagestyle{empty}
		\newtheorem{theorem}{Theorem}[section]
		\newtheorem{corollary}[theorem]{Corollary}
		\newtheorem{definition}{Definition}
		\newtheorem{guess}{Conjecture}
		\newtheorem{claim}[theorem]{Claim}
		\newtheorem{problem}{Problem}
		\newtheorem{question}{Question}
		\newtheorem{lemma}[theorem]{Lemma}
		\newtheorem{proposition}[theorem]{Proposition}
		\newtheorem{observation}[theorem]{Observation}
		\newenvironment{proof}{\noindent {\bf
				Proof.}}{\hfill\rule{3mm}{3mm}\par\medskip}
		\newcommand{\remark}{\medskip\par\noindent {\bf Remark.~~}}
		\newcommand{\ch}{{\rm ch}}
		\newcommand{\de}{\em}
		\newtheorem{example}{Example}
		\newcommand{\set}[1]{\{#1\}}
		\newcommand{\norm}[1]{{|#1|}}
		\newcommand{\NN}{\mathbb{N}}

		\newcommand{\GS}{G/\mathcal{S}}
		\newcommand{\LS}{L_{\mathcal{S}}}
		\newcommand{\BS}{B_{\mathcal{S}}}
		\newcommand{\XS}{X_{\mathcal{S}}}
		\newcommand{\YS}{Y_{\mathcal{S}}}
		\newcommand{\mS}{\mathcal{S}}
		\newcommand{\MS}{M_{\mathcal{S}}}
		
		\title{\bf Minimum non-chromatic-choosable graphs with given chromatic number}
		
		\author{Jialu Zhu\thanks{Department of Mathematics, Zhejiang Normal University, Email: jialuzhu@zjnu.edu.cn, }         \and
			Xuding Zhu\thanks{Department of Mathematics, Zhejiang Normal University, Email: xdzhu@zjnu.edu.cn, Grant numbers:   NSFC 12371359, U20A2068.}}
		
		\maketitle
		
		
		\begin{abstract}
			A graph $G$ is called chromatic-choosable if $\chi(G)=ch(G)$. A natural problem   is to determine the minimum number of vertices in a  non-chromatic-choosable graph with given chromatic number. 	It was conjectured by Ohba, and proved by Noel, Reed and Wu  that $k$-chromatic graphs $G$ with $|V(G)| \le 2k+1$ are chromatic-choosable. This upper bound on $|V(G)|$ 
			is tight. It is known that if $k$ is even, then  $G=K_{3 \star (k/2+1), 1 \star (k/2-1)}$ and $G=K_{4, 2 \star (k-1)}$ are non-chromatic-choosable $k$-chromatic graphs with $|V(G)| =2 k+2$. 
			Some   subgraphs of these  two   graphs are also non-chromatic-choosable. 
			The main result of this paper is that all other $k$-chromatic  graphs $G$ with $|V(G)| =2 k+2$ are chromatic-choosable. In particular, if $\chi(G)$ is odd and $|V(G)| \le 2\chi(G)+2$, then $G$ is chromatic-choosable, which was conjectured by Noel. 
		\end{abstract}
		
		

		\section{Introduction}
  \label{sec-intro}
		
		A \emph{proper colouring} of a graph  $G$ is a mapping $\phi:V(G) \rightarrow \mathbb{N}$ such that $\phi(u) \neq \phi(v)$ for every edge $uv$ of $E(G)$. A {\em  $k$-colouring } of $G$ is a proper colouring of $G$ using colours from $[k] = \{1,2,\ldots, k\}$. We say $G$ is {\em $k$-colourable} if there is a $k$-colouring of $G$.  
		The {\em chromatic number } $\chi(G)$ of $G$ is the minimum $k$ such that $G$ is $k$-colourable. 

		List colouring is a natural generalization of classical graph colouring, introduced independently by Erd\H{o}s-Rubin-Taylor \cite{ERT1980} and Vizing \cite{Vizing76} in 1970's.
		A {\em list assignment } of  $G$ is a mapping $L$ which
		assigns to each vertex $v$ a set $L(v)$ of permissible colours. 
		An {\em
			$L$-colouring} of $G$ is a proper   colouring $\phi$  of $G$ with $\phi(v) \in L(v)$ for each vertex $v$. We say that $G$ is {\em
			$L$-colourable} if there exists an $L$-colouring of $G$, and $G$
		is {\em $k$-choosable} if $G$ is $L$-colourable for any list assignment $L$ of $G$ with $|L(v)| \ge k$ for each vertex $v$. More
		generally, for a function  $g: V(G) \to \NN$, we say $G$ is   {\em
			$g$-choosable} if $G$ is $L$-colourable for every list assignment  $L$ with $|L(v)| \ge g(v)$ for all $v \in V(G)$.  
		The {\em choice number} $\ch(G)$ of $G$ is
		the minimum $k$ for which $G$ is $k$-choosable.
		
		A $k$-colouring of a graph $G$ is a special case of list colouring, where each vertex $v$ has the same list $L(v)=\{1,2,\ldots, k\}$. So $k$-choosable implies $k$-colourable.  At first glance, one might expect the reverse inequality to hold as well. The smaller intersection between lists  would make it easier to assign distinct colours to adjacent vertices. However, the reverse inequality is far from true. It was observed in \cite{ERT1980} and \cite{Vizing76} that for any integer $k$,  there are bipartite graphs that are not $k$-choosable. So the difference $ch(G)-\chi(G)$ can be arbitrarily large.

		A graph $G$ is called {\em chromatic-choosable} if $\chi(G)=ch(G)$.  
		Chromatic-choosable graphs have been studied a lot in the literature, and are related to some other difficult problems. 
		For example, the famous Dinitz  problem (see e.g. \cite{Zeilberger}) asks the following question:
		
		“Given an $n \times n$ array of $n$-sets, is it always possible to choose one from each set, keeping the chosen elements distinct in every row, and distinct in every column?” 
		
		This problem can be equivalently stated as whether the line graph of $K_{n,n}$ is chromatic-choosable?  This problem was solved by Galvin \cite{Gal}, who  proved a  more general result:  the  line graph  of any bipartite multigraph is chromatic-choosable. On the other hand, Galvin's result is a special case of a more general conjecture - the list colouring conjecture: line graphs of all multigraphs are chromatic-choosable. The list colouring conjecture was posed independently by many diﬀerent
		researchers: Albertson and Collins, Bollob\'{a}s and Harris, 
		Gupta, and Vizing (see \cite{BH1985,HC1992,JT1995}). It has attracted a lot of attention and remains open in general. 
		
		%
		%

		Ohba conjecture is another well-known conjecture about chromatic-choosable graphs.   It was proved in   \cite{Ohba2002}   that for any graph $G$,   $ch(G \vee K_n) = \chi(G \vee K_n)$ for   sufficiently large $n$, where $G \vee H$ is the join of $G$ and $H$, i.e., the graph obtained from the disjoint union of $G$ and $H$ by adding edges connecting every vertex of $G$ to every vertex of $H$.   This means that 
		graphs $G$ with $|V(G)|$ ``close'' to $\chi(G)$ are chromatic-choosable. A natural problem is how close should be $|V(G)|$ and $\chi(G)$ to ensure that $G$ be chromatic-choosable. Equivalently, what is the
		minimum number of vertices in a non-$k$-choosable $k$-chromatic graph? 
		
		We denote by $K_{k_1\star n_1,k_2 \star n_2, \ldots, k_q \star n_q}$ the complete  multi-partite graph with $n_i$ parts of size  $k_i$, for $i=1,2,\ldots, q$.  If $n_j=1$, then the number $n_j$ is omitted from the notation.
		It was proved in \cite{EOOS2002} that if $k$ is an even integer, then $K_{4, 2 \star (k-1)}$ and $K_{3 \star (k/2+1), 1 \star (k/2-1)}$ are  not   $k$-choosable. These two graphs are $k$-chromatic graphs with $2k+2$ vertices. 
		Ohba \cite{Ohba2002} conjectured that for any positive integer $k$,   $k$-chromatic graphs with at most $2k+1$ vertices are $k$-choosable. This conjecture has attracted considerable attention, and many partial results were proved before it was finally confirmed by Noel, Reed and Wu \cite{NRW2015}.

		One approach has been to prove variants of Ohba's conjecture in which $|V(G)| \le 2k+1$  is replaced by  $|V(G)| \le f(\chi(G))$ for some function $f$ with $f(k) < 2k+1$.
		Ohba \cite{Ohba2002} proved such a variant with $f(k)=k+\sqrt{k}$, and  Reed and Sudakov \cite{RS2005} improved the result to $f(k)= \frac 53 k - \frac 43$. By using a sophisticated  probabilistic method, Reed and Sudakov \cite{RS2002}  proved that Ohba's conjecture is asymptotically true:  if $|V(G)| \le (2-o(1))\chi(G)$, then $G$ is chromatic-choosable. 
		
		Another approach has been to show  the conjecture holds for special families of graphs. He, Li, Shen  and Zheng \cite{SHZL2009} proved Ohba's conjecture for graphs $G$ with independence number $\alpha(G) \le 3$, by extending a result of Ohba \cite{Ohba2004} who proved that if $|V(G)| \le 2\chi(G)$ and $\alpha(G) \le 3$, then $G$ is chromatic-choosable.  Kostochka, Stiebitz and Woodall \cite{KSW2011} improved this result and showed that Ohba’s Conjecture holds for graphs $G$ with $\alpha(G) \le 5$. Also Ohba's conjecture were verified for some particular complete multipartite graphs  in \cite{HZCSZ2008,SHZL2009,SHZWZ2008}.

		In 2015, Ohba's conjecture  was finally confirmed by Noel, Reed and Wu \cite{NRW2015}:  
		\begin{theorem}[Noel-Reed-Wu Theorem]  Every $k$-colourable  graph with at most $2k+1$ vertices is $k$-choosable.
		\end{theorem}
		
		Nevertheless, this is not the end of the story. More problems related to Ohba's conjecture are posed and studied. One problem is what would be the choice number of $k$-chromatic graphs $G$ with $|V(G)|$ slightly bigger than $2k+1$. This question was addressed in \cite{NWWZ2015}. Another related problem is the online version of Ohba's conjecture, which was posed in \cite{HWZ}, and has been studied in a few papers \cite{CCGH,KKLZ,KMZ2014}. Some partial cases are verified and the conjecture  remains  open in general.

		This paper explores the tightness of Ohba's conjecture.  Although Ohba's conjecture is tight,  $K_{4, 2\star (k-1)}$ and $K_{3 \star (k/2+1), 1 \star (k/2-1)}$ for even $k$ are the only known $k$-chormatic graphs with $2k+2$ vertices that are not $k$-choosable.  In particular, Ohba's conjecture was not known to be tight for odd integer $k$.
		
		Noel \cite{Noel2013} conjectured if $k$ is odd, then all $k$-chromatic graphs with $2k+2$ vertices are $k$-choosable.

		
		Observe that for a $k$-chromatic graph $G$, by adding edges between vertices of distinct colour classes, the resulting graph has the same chromatic number, and whose choice number is not decreased. Therefore in the study of minimum non-chromatic choosable graphs, it suffices to consider complete multipartite graphs. 
		
		The main result of this paper is that  $K_{4, 2\star (k-1)}$ and $K_{3 \star (k/2+1), 1 \star (k/2-1)}$  for even $k$ are the only non-$k$-choosable complete $k$-partite graphs with $2k+2$ vertices.

		\begin{theorem}
			\label{thm-main}
			Assume $G=(V,E)$ is a complete $k$-partite graph with  $|V| \le 2k+2$, and $G \ne K_{4, 2\star (k-1)},  K_{3\star (k/2+1), 1 \star (k/2 -1)}$ when $k$ is even, and $L$ is a $k$-list assignment of $G$. Then $G$ is $L$-colourable.  
		\end{theorem}

		As a consequence,   Noel's conjecture
		is confirmed: 
		
		\begin{corollary}
			\label{cor1}
			If $k$ is odd, then every $k$-chromatic graph with at most $2k+2$ vertices is chromatic-choosable. 
		\end{corollary}

		For a positive integer $k$, let 
		$$\beta(k) = \min\{|V(G)|: \chi(G)=k < ch(G)\}.$$
		
		For an odd integer $k$, it can be checked that $K_{5,2\star(k-1)}$ is not $k$-choosable. 
		Thus we have the following corollary.
		
		\begin{corollary}
			\label{cor2}
			For the function $\beta$ defined above, 
			\[
			\beta(k)=\begin{cases} 2k+2, &\text{ if $k$ is even}, \cr
				2k+3, &\text{ if $k$ is odd}.
			\end{cases}
			\]
		\end{corollary}

		Here is a brief outline of the proof of Theorem \ref{thm-main}.
		
		Assume $G$ is a complete $k$-partite graph with $2k+2$ vertices, $G \ne K_{4, 2 \star (k-1)}, K_{3 \star (k+1)/2, 1 \star (k-1)/2}$ when $k$ is even, and $L$ is a $k$-list assignment of $G$.
		Let $C_L= \bigcup_{v \in V}L(v)$. 
		The first step is to  construct a family  $\mathcal{S}$  of independent sets that form a partition of $V(G)$.
		Let $G/\mathcal{S}$ be the graph obtained from $G$ by identifying each independent set $S \in \mathcal{S}$ into a single vertex $v_S$.
		Let $L_{\mathcal{S}}$ be the list assignment of $G/\mathcal{S}$ defined as $L_{\mathcal{S}}(v_S) = \bigcap_{ u \in S}L(u)$.
		Build a bipartite graph $B_{\mathcal{S}}$ with partite sets $V(G/\mathcal{S})$ and $C_L$, with $\{v_S, c\}$ be an edge if $c \in L_{\mathcal{S}}(v_S)$. If $B_{\mathcal{S}}$ has a matching $M$ that covers $V(G/\mathcal{S})$, then $M$ defines an $L$-colouring of $G$, with each $S \in \mathcal{S}$ be coloured  with the colour matched to $v_S$ in $M$. 
		
		Assume that there is no such a matching $M$, and hence by Hall's Theorem, there exists a subset $X_{\mathcal{S}}$ of $V(G/\mathcal{S})$ such that $|Y_{\mathcal{S}}| < |X_{\mathcal{S}}|$, where $Y_{\mathcal{S}} = N_{B_{\mathcal{S}}}(X_{\mathcal{S}})$.  
		By analysing the lists $L(v)$ and independent sets $S$ in $\mathcal{S}$, 
		the inequality $|Y_{\mathcal{S}}| < |X_{\mathcal{S}}|$ may lead to a series of inequalities and eventually  lead to a contradiction (which means that no such $X_{\mathcal{S}}$ exists and hence the desired matching $M$ exists).
		
		Assume no contradiction is derived, and hence $X_{\mathcal{S}}$ and $Y_{\mathcal{S}}$ do exist. 
		We choose $\XS$ so that $|\XS|-|\YS|$ is maximum. 
		By Hall's Theorem, this implies that there is a matching $M'$ in $\BS-(\XS \cup \YS)$ that covers $V(G/\mathcal{S}) - \XS$.
		
		\begin{definition}
			\label{def-partial}
			A {\em partial $L$-colouring} of $G$ is an $L$-colouring of an induced subgraph $G[X]$ of $G$. Given an $L$-colouring $\phi$ of $G[X]$, 
			$L^{\phi}$ is the list assignment of $G-X$ defined as $L^{\phi}(v) = L(v) - \phi(N_G(v) \cap X)$ for $v \in V(G-X)$. An  $L$-colouring $\phi$ of $G[X]$
			is a {\em good partial $L$-colouring} of $G$ if the pair $(G-X, L^{\phi})$ satisfies the condition of Theorem \ref{thm-main}.
		\end{definition} 
		
		The matching  $M'$ constructed above  defines a {\em partial $L$-colouring $\psi$ of $G$} that colours vertices in $\bigcup_{S \in V(\GS)-\XS}S$. 
		One nice property of this partial colouring $\psi$ is that if $\{v\} \in \XS$ is a singleton part of $\mathcal{S}$, then $L^{\psi}(v) = L(v)$ (as $L(v) \subseteq \YS$). In other words some neighbours of $v$ may have been coloured, and yet $v$  still has the same set of permissible colours. 
		
		By using this property, we want to  extend  $\psi$ to a good partial $L$-colouring $\phi$ of $G$, that colours a subset $X$ of $G$. If this can be done, then $G-X$ has an $L^{\phi}$-colouring $\theta$, and the union $\phi \cup \theta$ would be an $L$-colouring of $G$. 
		
		For the plan above to work, the choice of the 
		partition $\mathcal{S}$ of $V(G)$ in the first step is crucial. Indeed, Theorem \ref{thm-main} is equivalent to saying that there is a choice of $\mathcal{S}$ such that $B_{\mathcal{S}}$ has a matching $M$ that covers $V(G/\mathcal{S})$. 
		We usually start with a proper colouring $f$ of $G$, which is not necessarily an $L$-colouring, but ``close'' to an $L$-colouring, and let $\mathcal{S}$ be the colour classes of $f$.  The concept of ``near acceptable'' $L$-colouring is defined to capture the required properties needed for the plan above to work.  Near acceptable $L$-colouring was first used in \cite{NRW2015}. The definitions of near acceptable $L$-colourings for the proofs of Noel-Reed-Wu Theorem and Theorem \ref{thm-main} are slightly different. The slight difference makes it more difficult to construct  a near acceptable $L$-colouring of $G$ for the proof of Theorem \ref{thm-main}, while the proof of Noel-Reed-Wu Theorem is already complicated. For the proof of Theorem \ref{thm-main}, before constructing a near acceptable $L$-colouring of $G$, a pseudo-$L$-colouring of $G$ is constructed as an intermediate step. In many cases, we need to repeatedly modify a pseudo $L$-colouring until we obtain a near acceptable $L$-colouring.

		In Section 2, we prove a sufficient condition for a complete multipartite graph $G$ with all parts of size at most 3 to be $g$-choosable for a given function $g:V(G) \to \mathbb{N}$. This will be used in later proofs. In Section3, we fix some notation and present some basic properties of a minimum counterexample. 
		In Section 4, we prove Theorem \ref{thm-main} for complete $k$-partite graphs with most parts of size at most 3. These graphs are special as there is little difference between these graphs and the critical graphs   $K_{4, 2 \star (k-1)}$ and $K_{3 \star (k/2+1), 1 (k/2-1)}$ (for even $k$).
		In Section 5, we introduce the concept of pseudo-$L$-colouring of $G$ and prove some properties of such colourings. 
		In Section 6, we  define the concept of near-acceptable $L$-colouring
		and show that the existence of a near-acceptable $L$-colouring of $G$ implies the existence of a proper $L$-colouring of $G$. 
		Some sufficient conditions for the existence of  near-acceptable $L$-colourings of $G$ are presented in Sections 7
		and 8. A final contradiction is derived in Section 9.


		\section{Graphs with all parts of sizes at most 3}
		
		This section proves the following lemma, which  gives a sufficient condition for $g: V(G) \to \mathbb{N}$, so that $G$ is $g$-choosable when all parts of $G$ have size at most 3. This lemma is analog to Lemma 5 in \cite{KMZ2014}, where a sufficient condition for $G$ to be on-line $g$-choosable was given. The sufficient condition below is almost the same as that in Lemma 5 of \cite{KMZ2014}, except that for two vertices $u,v$ in a 3-part of $G$,  the upper bounds for the sum $g(u)+g(v)$ in the two lemmas   are different, and which is needed in later applications. 

		\begin{lemma}
			\label{ind3}
			Let $G$ be a complete multipartite graph with parts of size at most $3$. Let  $\mathcal{A},$   $\mathcal{B}$, $\mathcal{C}$, $\mathcal{D}$ be a partition of the   parts of $G$ into classes  such that $\mathcal{A}$ and $\mathcal{D}$ contain only parts of size $1$,   $\mathcal{B}$ contains all parts of size $2$  and $\mathcal{C}$ contains all parts of size $3$. Let $k_1, k_2, k_3, d$ denote the cardinalities of classes $\mathcal{A}$,  $\mathcal{B}$, $\mathcal{C}$,  $\mathcal{D}$ respectively. Suppose that classes $\mathcal{A}$ and $\mathcal{D}$ are ordered, i.e.\ $\mathcal{A}= (A_1, \ldots, A_{k_1})$
			and $\mathcal{D}=(D_1, \ldots, D_d)$. If $f:V(G)\to \NN$ is a function for which the following   hold:
			\begin{alignat}{2}
				g(v)&\geq k_2+k_3+i, &&\quad\text{for all $1\leq i \leq k_1$ and $v\in A_i$}\tag{a-1}\label{invK:a-1}\\
				g(v)&\geq k_2+k_3, &&\quad\text{for all $v\in B \in\mathcal{B}$}\tag{b-1}\label{invK:b-1}\\
				g(u)+g(v)&\geq 3k_3+2k_2+k_1+d, &&\quad\text{for all $u,v\in B\in\mathcal{B}$}\tag{b-2}\label{invK:b-2}\\
				g(v)&\geq k_2+k_3, &&\quad\text{for all $v\in C\in\mathcal{C}$}\tag{c-1}\label{invK:c-1}\\
				g(u)+g(v)&\geq 2 k_3 + 2k_2+ k_1, &&\quad\text{for all $u,v\in C\in\mathcal{C}$}\tag{c-2}\label{invK:c-2}\\
				\sum_{v\in C}g(v)&\geq 4 k_3+ 3  k_2+2k_1+d-1, &&\quad\text{for all $C\in\mathcal{C}$} \tag{c-3}\label{invK:c-3}\\
				g(v)&\geq 2 k_3+k_2+k_1+i, &&\quad\text{for all $1\leq i \leq d$ and $v\in D_i$}\tag{d-1}\label{invK:s-1}
			\end{alignat}
			then $G$ is $g$-choosable.	
		\end{lemma}
		
		\begin{proof}
			Assume the parts of $G$ are partitioned into $\mathcal{A}$, $\mathcal{B}$,  $\mathcal{C}$,   $\mathcal{D}$   and  $g$ is a function satisfying the inequalities \eqref{invK:a-1}-\eqref{invK:s-1}, and $L$ is a list assignment with   $\norm{L(v)}= g(v)$.  We shall colour an independent set $S$ of $G$ with a 
			colour $c \in \bigcap_{v \in S}L(v)$. Let $G'=G-S$ and $L'$ be the list assignment of $G'$ defined as $L'(x) = L(x) -\{c\}$ for $x \in V(G')$ and $g'(v)=|L'(v)|$.
			We shall verify that the pair $(G', f')$ satisfies the condition of Lemma \ref{ind3}, and hence $G'$ is $L'$-colourable by induction hypothesis (if $|V(G)|=1$, then the result is trivial). 
			Together with the colouring of $S$ with colour $c$, we obtain an $L$-colouring of $G$. 
			
			In the following, we describe the choice of the independent set $S$. The colour $c$ is always an arbitrary colour in  $\bigcap_{v\in S}L(v)$.
			We describe briefly how to verify the fact that $(G', g')$ satisfies the condition of Lemma \ref{ind3} (the proof of Lemma 5 of \cite{KMZ2014} is similar, and contains more detailed explanations).  The partition $\mathcal{A}'$, $\mathcal{B}'$,  $\mathcal{C}'$,   $\mathcal{D}'$ of the parts of $G'$ and the ordering of parts in $\mathcal{A}'$ and $\mathcal{D}'$ are inherited from the partition and the ordering of the parts of $G$, except that one part may have some vertices coloured and remaining vertices form a part in another class.  When a part from $\mathcal{B}$ or $\mathcal{C}$ has some vertices coloured and the remaining vertex form a part in $\mathcal{A}'$ or $\mathcal{D}'$, we also need to put it in a correct order. Denote by $k'_1,k'_2,k'_3, d'$ the cardinalities of $\mathcal{A}'$, $\mathcal{B}'$,  $\mathcal{C}'$,   $\mathcal{D}'$, respectively. To verify the inequalities, it suffices to show that with $g$ replaced by $g'$,  $k_i$ replaced by $k'_i$ and $d$ replaced by $d'$, the amount reduced on the left hand side is no more than the amount reduced on the right hand side. 
			
			The choice of $S$ is determined in 8 cases.  For $ 2 \le i \le 8$, Case $i$ is considered only if all cases $j$ with $j \le i-1$ do not apply.

			\begin{enumerate}
				\item If there exists $C \in \mathcal{B} \cup \mathcal{C}$ for which  $\bigcap_{v\in C}L(v) \ne \emptyset$, then $S=C$. \\
				{\em Verification: For (a-1),(b-1),(c-1), (d-1), the left hand side is reduced by at most 1 (i.e., $g'(v) \ge g(v)-1$), and the right hand side is reduced by at least 1. (For example, consider (a-1):  $k'_2+k'_3+i = k_2+k_3+i - 1$).  For (b-2), (c-2),  the left hand side is reduced by at most 2 (i.e., $g'(u)+g'(v) \ge g(u)+g(v)-2$), and the right hand side is reduced by at least 2. For (c-3),  the left hand side is reduced by at most 3 (i.e., $\sum_{v \in C}g'(v) \ge \sum_{v \in C}g(v)-3$), and the right hand side is reduced by at least 3.}
				
				\item If there exist $ C=\{u,v, w\} \in\mathcal{C}$ with $g(u)+g(v)= 2 k_3 +2k_2+ k_1$, and  $L(u)\cap L(v) \ne  \emptyset$, then $S=\{u,v\}$. \\
				{\em Verification: The part $\{w\}$ of $G'$ is the last member of $\mathcal{D}'$. Thus $k'_3=k_3-1$ and $d'=d+1$. Note that $g'(w)=g(w) \ge 4k_3+3k_2+2k_1+d-1 - (2k_3+2k_2+k_1) = 2k_3+k_2+k_1+d-1 = 2k'_3+k'_2+k'_1+d'$. The other inequalities are verified as  in Case 1.}
				
				\item If there exists $C=\{v,u,w\} \in\mathcal{C}$, $g(v)=k_2+k_3$, $L(v) \cap L(u) \neq \emptyset$,  then $S=\{u, v\}$.\\
				{\em Verification: The part $\{w\}$ of $G'$ is the last member of $\mathcal{A}'$. Thus $k'_3=k_3-1$ and $k'_1=k_1+1$. Note that $g'(w)=g(w) \ge 2k_3+2k_2+k_1 - (k_3+k_2) = k_3+k_2+k_1 = k'_3+k'_2+k'_1$. For $u,v \in C  \in \mathcal{C}$, either $g(u)+g(v) \ge 2k_3+2k_2+k_1+1$ or $g'(u)+g'(v) \ge g(u)+g(v)-1$ (as Case 2 does not apply). Hence (c-2) holds for $(G',g')$.  As Case 1 does not apply, the left hand side of (c-3) reduces by at most 2, and the right hand side is reduced by 2. Hence (c-3) holds for $(G',g')$ as Case 1 does not apply. The other inequalities are verified as  in Case 1.}

				\item If there exists $C=\{v,u,w\} \in\mathcal{C}$,  $g(v)=k_2+k_3$, $L(v) \cap (L(u) \cup L(w)) = \emptyset$, then $S=\{v\}$. \\
				{\em Verification: In the remaining graph $G'=G-v$, the two vertices $u,w$ are identified into a single vertex $u^*$ with $L'(u^*) = L(u) \cap L(w)$. The set $\{u^*\}$ is the last member of $\mathcal{A}'$. So $k'_3=k_3-1, k'_1=k_1+1$. Note that 
					\[
					g(u)+g(w)\geq (4k_3+3k_2+2k_1+d-1)-(k_3+k_2)=3k_3+2k_2+2k_1+d-1.
					\]
					On the other hand the total number of colours is at most $|V|-1 = 3 k_3 +2k_2+ k_1 +d-1$. As $L(v)$ is disjoint with $L(u)\cup L(w)$,  we have $\norm{L(u ) \cup L( w)} \leq 2 k_3 +k_2+k_1+d -1$. Hence
					\[
					|L'(u^*)| = \norm{L(u) \cap L(w)} \geq k_3 + k_2+k_1 =k'_3+k'_2+k'_1.
					\]
					Note that for $C \in \mathcal{C}$, $\sum_{v \in C}g'(v) \ge \sum_{v \in C}g(v)-2$, as Case 1 does not apply. Hence (c-3) holds for $(G',g')$. The other inequalities are verified as  in Case 3.}

				\item If there exists $ B=\{u,v\} \in\mathcal{B}$, $g(v)=k_2+k_3$, then $S=\{v\}$. \\
				{\em Verification: The part $\{u\}$ of $G'$ is  the last member of $\mathcal{D}'$. Thus $k'_2=k_2-1$ and $d'=d+1$. Note that $g'(u)=g(u) \ge 3k_3+2k_2+k_1+d - (k_3+k_2) = 2k_3+k_2+k_1+d = 2k'_3+k'_2+k'_1+d'$. For $B'=\{x,y\} \in \mathcal{B}$, since Case 1 does not apply, $g'(x)+g'(y) \ge g(x)+g(y)-1$. 
					So (b-2) holds for $(G',g')$.  The other inequalities are verified as  in Case 4. }

				\item If $k_1 \ne 0$ and $A_1=\{v\}$, then $S=\{v\}$.\\
				{\em Verification: In this case, $k'_1=k_1-1$.  As Cases 2,3,4 do not apply,   (b-1), (c-1) and (c-2) were not tight for $g$, and hence they hold for $(G',g')$. Also for (a-1), the index of each member reduces by 1, and hence the right hand side reduces by 1, so it holds for $(G',g')$.  The other inequalities are verified as  in Case 5.}
				
				\item Assume $k_3 \ne 0$
				and $C=\{u,v,w\} \in \mathcal{C}$. As $|C_L| \le |V|-1= 3k_3+2k_2+k_1+d-1$, 
				So $g(u)+g(v)+g(w) \ge 4k_3+3k_2+2k_1+d-1>|C_L|$ and there is  a colour $c$ which appears in two of the three colour sets $L(u)$, $L(v)$, $L(w)$, say $c\in L(u)\cap L(v)$. Let $S=\{u,v\}$. \\
				{\em Verification:  Let $\{w\}$ be the only member of $\mathcal{A}'$. Then $k'_3 = k_3-1$ and $k'_1=k_1+1$, $g'(w)=g(w) \ge k_2+k_3=k'_2+k'_3+1=k'_2+k'_3+k'_1$.  The other inequalities are verified as  in Case 6.}

				\item If $ d > 0$ and $D_1=\{v\}$,  then $S=\{v\}$.\\
				{\em Verification:  In this case,  $k_3=k_1=0$ and $d'=d-1$. (b-1) is not tight for $g$ (as Case 5 does not apply), and hence holds for $(G',g')$. (b-2) holds for $(G',g')$ as the left-hand size reduces by at most 1, and the right hand side reduces by 1. For other member of $\mathcal{D}'$, its index is recued by 1, and hence (d-1) holds for $(G',g')$. Note that $k_1,k_3=0$, so the other inequalities are vacant.}
			\end{enumerate}
			
			Assume all the cases above do not apply.  Then  $G=K_{2 \star k_2 }$, i.e., $G$ consists of $k_2$ parts of size 2, and $g(v) \ge k_2$ for each vertex $v$. It is well-known \cite{ERT1980} that in this case, $G$ is $g$-choosable.
		\end{proof}
		
		\section{Some notation and basic properties for a minimum counterexample}

 By a counterexample of Theorem \ref{thm-main}, we mean a pair $(G,L)$ such that $G$ is a complete multipartite graph and $L$ is a list assignment of $G$ that satisfy the condition of Theorem \ref{thm-main}, and $G$ is not $L$-colourable. 
We say $(G,L)$ is a minimal counterexample to Theorem \ref{thm-main} if $(G,L)$ is a counterexample to Theorem \ref{thm-main} with 
		\begin{enumerate}
			\item  $|V(G)|$ minimum,
			\item  subject to (1), with $|C_L|$ minimum (recall that $C_L=\bigcup_{v \in V}L(v)$),
		\end{enumerate}
		It is well-known \cite{Kierstead} that   $|C_L| < |V(G)|$. 	Let 
		\begin{equation}
			\label{eq-lambda}
			\lambda = |V|-|C_L|>0.  
		\end{equation}

  In the remainder of this paper, we assume that $(G,L)$ is a minimum counterexample to Theorem \ref{thm-main}. Assume $G$ is a complete $k$-partite graph.
		By Noel-Reed-Wu Theorem, we know that $k$-chromatic graphs with at most $2k+1$ vertices are $k$-choosable and hence $G$ has exactly $2k+2$ vertices, and  \begin{equation}
			\label{eq-cl}
			|C_L| \le 2k+1.
		\end{equation}
		
		A   part of $G$ of size $i$ (respectively, at least $i$ or at most $i$) is called a {\em $i$-part} (respectively, {\em  $i^+$-part}, or {\em $i^-$-part}). Let $$T = \{v: \{v\} \text { is  a singleton part of  } G\}.$$
		Let $p_i$,  $p_i^+$ and $p_i^-$ be the number of $i$-parts,  $i^+$-parts and $i^-$-parts, respectively.

		For a subset $X$ of $V(G)$, let 
		$$L(X) = \bigcup_{v \in X} L(v).$$
		For three vertices $x,y,z$ of $G$, let 
		$$L(x \vee y) = L(x) \cup L(y), L(x \wedge y) = L(x) \cap L(y),$$
		$$ L((x \wedge y) \vee z) = (L(x) \cap L(y)) \cup L(z).$$

		For $c \in C_L$ and $C' \subseteq C_L$, let $$L^{-1}(c)=\{v: c \in L(v)\}, \ L^{-1}(C') = \bigcup_{c \in C'}L^{-1}(c).$$

		For a part $P$ of $G$ and integer $i$, let 
		\begin{eqnarray*}
			C_{P,i} &=& \{c \in C: |L^{-1}(c) \cap P| = i \}, \\
			\Lambda_{P,i} &=& \max\{| \bigcap_{v \in S}L(v)|: S\subseteq P, |S|=i\}.
		\end{eqnarray*}.

		Assume $\mathcal{S}$  is a partition of $V(G)$
		into a family of  independent sets. Each $S \in \mathcal{S}$ is called an {\em  $\mathcal{S}$ part}. Recall that $\GS$ is the graph obtained from $G$ by identifying each part $S \in \mathcal{S}$ into a single vertex $v_S$, and $\LS$ is the list assignment of $G/\mathcal{S}$ defined as $\LS(v_S)=\bigcap_{v \in S}L(v)$.
		If $S=\{v\}\in \mathcal{S}$ consists of a single vertex of $G$, then we denote   $v_S$ by $v$.
		In this case, $\LS(v)=L(v)$. For the partitions $\mS$ constructed in this paper, most parts of $\mS$ are singletons. To define $\mS$, it suffices to list  its non-singleton parts.

		Recall that $\BS$ is the bipartite graph with partite sets $V(\GS)$ and $C_L$, in which $\{v_S,c\}$ is an edge if and only if $c \in \LS(v_S)$. 
		A matching $M$ in $\BS$ covering $V(\GS)$ induces an $\LS$-colouring of $\GS$, which in turn induces an $L$-colouring of $G$. 	Since $G$ is not $L$-colourable, no such matching $M$ exists. By Hall's Theorem, there is a subset $\XS$ of $V(\GS)$ such that 
		$|\XS| > |N_{\BS}(\XS)|$.
		
		We denote by $\XS$ a subset of $V(\GS)$ for which $|\XS| - |N_{\BS}(\XS)|$ is maximum.
		Let $$\YS=	N_{\BS}(\XS)=\bigcup_{v_S \in \XS}\LS(v_S).$$ 
		The choice of $\XS$ implies that there is a matching $\MS$ in $\BS - (\XS \cup \YS)$   that covers all vertices in $V(\GS) - \XS$. The matching $\MS$ defines a partial colouring $\psi_{\mathcal{S}}$ of $G[\bigcup_{S \in \mathcal{S}-\XS}S]$ with colours from $C_L-\YS$.

		These notation will be used throughout the whole paper.
		
		\begin{observation}
			\label{obs-0} The following easy facts will be used often in the argument. 
			\begin{enumerate}
				\item There is an  injective mapping $\phi: C \to V$ such that $c \in L(\phi(c))$.
				\item If $f$ is a proper colouring of $G$, then there is a surjective  proper colouring $g:V\to C$ such that for every vertex $v$,  $g(v)\in L(v)$ or $g(v)=f(v)$.
				\item No two vertices in the same part of $G$ have the same list, and no colour is contained only  in the lists of vertices in  a same part.
     {\item $G\neq K_{4,2\star(k-1)}$ for any $k$ and $|T|\ge 1$.}
			\end{enumerate}
		\end{observation}
		\begin{proof}
			(1)  is well-known ({Corollary 1.8 in \cite{NRW2015}}) and also easy to verify (use the minimality of $|C_L|$). 
			
			(2) was   proved in {Proposition 1.13 in } \cite{NRW2015}.
			
			(3) If $u,v$ are in the same part and $L(u) = L(v)$, then 
			By Noel-Reed-Wu Theorem, there is a proper $L$-colouring $f$ of $G-u$,  which extends to   a proper $L$-colouring of $G$ by letting $f(u)=f(v)$.
			
			If there is a colour $c$ such that $L^{-1}(c) \subseteq P_i$ for some part $P_i$ of $G$, then by Noel-reed-Wu Theorem, $G-L^{-1}(c)$ has an $L$-colouring $f$, which extends to an $L$-colouring of $G$ by colouring vertices in $L^{-1}(c)$ with colour $c$.

    (4)  It was proved in \cite{EOOS2002}  that $K_{4,2\star(k-1)}$ is not $k$-choosable if and only if $k$ is even.   By our assumption, $G\neq K_{4,2\star(k-1)}$ for even $k$. 
    Thus $G\neq K_{4,2\star(k-1)}$ for any $k$. It was proved in \cite{GM1998} that $G=K_{3\star 2, 2 \star(k-2)}$ is $k$-choosable. Using the fact that $|V(G)|=2k+2$, it is easy to see that $|T| \ge 1$.
		\end{proof}

		\begin{lemma}
			\label{lem-noncommoncolour}
			If 	$P$ is a $2^+$-part  of $G$, then $\bigcap _{v\in P}L(v)=\emptyset$. Consequently for each colour $c \in C$, $|L^{-1}(c)| \le k+p_1+2$. 
		\end{lemma}
		\begin{proof}
			Assume the lemma is not true. 	  
			We choose such a
			part $P$ of maximum size, and colour vertices in $P$ by a common colour $c$. Let $L'(v) = L(v)-\{c\}$ for $v \in V(G)-P$. If $|P| \ge 3$, then $L'$ and $G-P$ satisfies the condition of Noel-Reed-Wu Theorem and hence $G-P$ has an $L'$-colouring. 
			
			Assume $|P|=2$. By (4) of Observation \ref{obs-0},  $G-P \ne K_{4, 2 \star (k-2)}$. 
			If  $G-P \ne  K_{3 \star (q+1), 1 \star (q-1)}$, then by the minimality of $G$, $G-P$ has an $L'$-colouring. 
			If $G-P =  K_{3 \star (q+1), 1 \star (q-1)}$, then since
			each 3-part $P$    has at most two vertices $v$ for which $c \in L(v)$, it is   straightforward to verify that $G-P$ and $L'$ satisfy the condition of Lemma \ref{ind3}. Hence $G-P$ has an $L'$-colouring.

			For any colour $c \in C$, each $2^+$-part contains a vertex $v \notin L^{-1}(c)$. So $$|L^{-1}(c)| \le |V(G)| - p_2^+ = 2k+2 - (k-p_1)=k+p_1+2.$$
			This completes the proof of Lemma \ref{lem-noncommoncolour}.
		\end{proof}



		\section{Graphs with  most parts of size at most 3}

		In	this section, we consider complete $k$-partite graphs  whose most parts  are $3^-$-parts. 
		
		Let 
		\begin{eqnarray*}
			\mathcal{G}_1 &=& \{ K_{5, 3 \star (q-1), 2 \star (k-2q), 1 \star q}: k \ge 2q \ge 2\}, \\
			\mathcal{G}_2 &=& \{K_{4 \star a, 3 \star (q-a), 2 \star b, 1 \star (k-q-b)}: a \le 2, {a \le q, b\ge 0, q+b\le k, a+2q+b=k+2.}\}
		\end{eqnarray*}

		\begin{theorem}
			\label{thm-main2}
			$G \notin  \mathcal{G}_1 \cup \mathcal{G}_2$.
		\end{theorem}

		We may assume that $k \ge 8$, as for $k \le 7$, we can check directly the graphs in $\mathcal{G}_1, \mathcal{G}_2$ are $k$-choosable. 
  
  Assume $G \in  \mathcal{G}_1 \cup \mathcal{G}_2$. 
		We order  the parts of $G$ as $ P_1,P_2, \ldots, P_k$ so that 
		\begin{itemize}
			\item if $G \in \mathcal{G}_1$, then   $P_1$ is 
			the 5-part and $P_2, P_3,\ldots , P_q$ are $3$-parts with $\Lambda_{P_2,2}\ge \Lambda_{P_3,2}\ge \ldots \ge \Lambda_{P_q,2}$;
			\item if $G \in \mathcal{G}_2$, then the first $a$ parts are the 4-parts of $G$, and 
			$P_{a+1}, P_{a+2},\ldots , P_q$ are $3$-parts with $\Lambda_{P_2,2}\ge \Lambda_{P_3,2}\ge \ldots \ge \Lambda_{P_q,2}$. If $a =2$, then order $P_1,P_2$ so that $\Lambda_{P_1,3} \ge \Lambda_{P_2,3}$.
		\end{itemize} 
		Let $$i_0=\max \{j: \Lambda_{P_j, 2}\ge j\}.$$
		For a 3-part $P $ of $G$,   we have $3k \le \sum_{v \in P}|L(v)| \le  |C_L|+|C_{P,2}|  \le 2k+1+|C_{P,2}|$. So $|C_{P,2}| \ge k-1$. As $P$ has three 2-subsets,  we have $\Lambda_{P, 2} \ge (k-1)/3 \ge 2$. 
		
		\begin{claim}
			\label{clm-p1}
			If $G\in \mathcal{G}_1$, then 	 $C_{P_1,4} = \emptyset$ and $C_{P_1,3} \ne \emptyset$.  
		\end{claim}
		\begin{proof}
			If $c \in C_{P_1,4}$, then   we colour vertices in $L^{-1}(c) \cap P_1$ with colour $c$, and let $L'(v)= L(v)-\{c\}$ for $v \in G-(L^{-1}(c) \cap P_1)$. It is easy to verify that $G'=G-(L^{-1}(c) \cap P_1)$ and $L'$ satisfy the condition  of Lemma \ref{ind3} (with $P_1-L^{-1}(c)$ being the last part in $\mathcal{A}$, and with $\mathcal{D} = \emptyset$), and hence $G'$ is $L'$-colourable, and $G$ is $L$-colourable, a contradiction.
			
			If  $C_{P_1,3} = \emptyset$, then each colour $c \in C_L$ is contained in $L(v)$ for at most two vertices $v \in P_1$. Hence $2(2k+1) \ge 2|C_L| \ge \sum_{v \in P_1}|L(v)| = 5k$, which implies that $k \le 2$, a contradiction. 
		\end{proof}
		
		\begin{claim}
			\label{cl-1b}
			$G \ne K_{5,2\star(k-2),1}$.
		\end{claim}	
		\begin{proof}
			If   $G=K_{5,2\star(k-2),1}$, then fix a 3-subset $S_1$ of $P_1$ with 
			$\bigcap_{v \in S_1}L(v) \ne \emptyset$. Let $\mathcal{S}$ be the partition of $0V(G)$ with one non-singleton part  $S_1$.  Then   $|V(\GS)|=2k$ and hence $|\XS|\le 2k$ and $|\YS| \le 2k-1$. By Lemma \ref{lem-noncommoncolour}, $|\XS \cap P|\le 1$ for any 2-part $P$. So
			$|\XS|\le k+2$ and $|\YS| \le k+1$.
			On the other hand,  $|\XS|\ge 2$ and hence $v\in \XS$ for some vertex $v$ with $|\LS(v)|\ge k$ and hence $|\YS| \ge k$ and $|\XS|\ge k+1$, and hence 
			$|\XS \cap P'_1 | \ge 2$. This in turn implies that   $|\YS|=k+1$ and hence $|\XS|=k+2$. Then   $P'_1\subseteq \XS$ and  $|\YS|\ge  |\LS(P'_1)| \ge k+2=|\XS|$ (by Claim \ref{clm-p1}), a contradiction.  
		\end{proof}

		 It follows from Observation \ref{obs-0} that  $G\neq K_{4,2\star(k-1)}$ for any $k$. As $G \ne K_{5,2\star(k-2),1}$,  $G$ has at least two $3^+$-parts. Therefore $$i_0 \ge  2.$$
		
		For $i=1,2,\ldots, i_0$, we shall  choose a subset $S_i$ of $P_i$ of size $2$ or $3$, and let $\mS$ be the partition of $V(G)$ with non-singleton parts $\{S_1,S_2, \ldots, S_{i_0}\}$. The rules for choosing the sets $S_i$ will be given later.
		
		For simplicity, in the graph $\GS$, for $i=1,2,\ldots, i_0$, we denote $v_{S_i}$ by $z_i$, and let 
		$$Z = \{z_1, z_2, \ldots, z_{i_0}\}.$$
		We denote by $P'_i$ the part of $\GS$, where for $1 \le i \le i_0$, $P'_i$ is  obtained from the part $P_i$   by identifying $S_i$ into a new vertex $z_i$, and for $i_0+1 \le i \le k$, $P'_i=P_i$. 
		
		As $i_0 \ge 2$, we have $|V(\GS)| \le 2k$, and hence
		\begin{equation}
			\label{eq-2k}
			|\XS| \le 2k, \ \ |\YS| \le 2k-1.
		\end{equation} 
		We shall prove further upper and lower bounds for $|\XS|$ and $|\YS|$ that eventually lead to a   contradiction.

		The details are delicate and a little complicated, which is perhaps unavoidable, as $K_{4, 2 \star (k-1)}$ and $K_{3 \star (k/2+1), 1 \star (k/2-1)}$ (for even integer $k$) are very close to graphs in $\mathcal{G}_1 \cup \mathcal{G}_2$, and they are   not $k$-choosable. We divide the proofs for $G \notin \mathcal{G}_1$ and $G \notin \mathcal{G}_2$ into two subsections. 
		
		
		\subsection{$G \notin \mathcal{G}_1$}
		
		Assume to the contrary that  $G \in \mathcal{G}_1$.

		The subsets $S_i$ for $i=1,2,\ldots, i_0$ are chosen as follows:
		\begin{enumerate}
			\item[(1)] $S_1$ is a 3-subset of $P_1$ with $|\bigcap_{v \in S_1}L(v)| =\Lambda_{P_1,3}$.  
			\item[(2)]  For $2 \le i \le i_0$,  $S_i$ is a 2-subset of $P_i$ with $|\bigcap_{v \in S_i}L(v)|=\Lambda_{P_i,2}$.
		\end{enumerate}

		Assume for $i=2,3,\ldots, i_0$, $P_i=\{u_i,v_i,w_i\}$ and $S_i=\{u_i, v_i\}$.

		Since $|P_1-S_1|=2 $, by (3) of Observation \ref{obs-0}, $|L(P_1-S_1)| \ge k+1$. As $(\bigcap_{v \in S_1} L(v)) \cap L(P_1-S_1) = \emptyset$, we know that 
		\begin{equation}
			\label{eq-p1}
			|\LS(P'_1)| \ge k+2.
		\end{equation}
		
		It follows from the definition of $\mS$ that 
		for $i=1,2,\ldots, i_0$, $|\LS(z_i)| \ge i_0$.
		
		If $\XS \subseteq Z$ and $z_i \in \XS$ for some $i \le i_0$, then we have $|\YS| \ge |\LS(z_i)| \ge i_0 \ge |\XS|$, a contradiction. Thus $\XS -Z \ne \emptyset$. Let $v \in \XS-Z$. Then 
		\begin{equation*}
			\label{eq-k}
			|\YS| \ge |\LS(v)| = |L(v)| \ge k, \ \ |\XS| \ge k+1.
		\end{equation*}

		This implies that  $|\XS \cap P'_i| \ge 2$ for some $i$. As 
		$|\LS(A)| \ge k+1$ for any 2-subset $A$ of $P'_i$ (for any $i$), we have   \begin{equation}
			\label{eq-k+1}
			|\YS| \ge  k+1, \ \ |\XS| \ge k+2.
		\end{equation}

		\begin{claim}
			\label{lem-ki0}
			$ |\YS| \ge k+i_0$ and hence $|\XS| \ge k+i_0+1$.
		\end{claim}
		\begin{proof}
			If there is an index $  i_0+1 \le i \le q$ such that $u,v \in \XS \cap P'_i$, then 
			$ |\YS| \ge |L(u \vee v)| \ge 2k-|L(u \wedge v)| \ge 2k-i_0 > k+i_0$ (as $i_0 \le q-1 < k/2$) and we are done. 
			
			Assume $ |\XS \cap P'_i|\le 1$ for any $  i_0+1 \le i \le q$. 
			If  $\{z_i, w_i\} \subseteq \XS$ for some $i \ge 2$,  then $|\YS| \ge |L(w_i)|+|L_S(z_i)| +\ge k+i_0$, and we are done. 
			Otherwise, $|\XS| \ge  k+2$ (by (\ref{eq-k+1})) implies that 
			$P'_1 \subseteq \XS$ and $|\XS| = k+2$. By (\ref{eq-p1}), $|\YS| \ge |\LS(P'_1)| \ge k+2$, a contradiction.
		\end{proof}
		
		\begin{claim}
			\label{lem-ki01}
			If
			$ |\YS| = k+i_0$, then $\Lambda_{P_i,2} = i_0$ for $i=2,3,\ldots, i_0$ and 
			there is an index $2 \le i \le i_0$  such that 
			$P_i$ has a  $2$-subset $S$  with $| \bigcap_{v \in S}L(v)| \ge 2$  and $\bigcap_{v \in S}L(v) \cup L(P_i-S) \not\subseteq \YS$.
		\end{claim}
		\begin{proof}
			Assume $|\YS|=k+i_0$. Then $|\XS| \ge k+i_0+1$.

			By the argument in the proof of  Claim \ref{lem-ki0}, for any index $i > i_0$, $|\XS \cap P_i| \le 1$.
			This implies that $|\XS| \le k+i_0+1$, and hence $|\XS| = k+i_0+1$ and $P'_i \subseteq \XS$ for $i=1,2,\ldots, i_0$. As $|\LS(P'_i)| \ge k+i_0$ for $2 \le i \le i_0$, we conclude that for $2 \le i \le i_0$, $\YS = \LS(P'_i)$ and $\Lambda_{P_i,2}=i_0$. 
			
			We shall find an index $2 \le i \le i_0$,  a  $2$-subset $S$ of $P_i$ with $| \bigcap_{v \in S}L(v)| \ge 2$  and $\bigcap_{v \in S}L(v) \cup L(P_i-S) \not\subseteq \YS$.
			
			Assume first that  there is an index $2 \le i \le i_0$ such that $L(P_i) \not\subseteq \YS$. 
			
			As  $L(w_i)\subseteq \YS$,
			we may assume that there is a colour  $c\in L(u_i)-\YS$. 
			If $|L(v_i \wedge w_i)|\ge 2$, then let $S=\{v_i,w_i\}$, we are done. 
			
			Assume $|L(v_i \wedge w_i)|\le 1$. This implies that $  |L(v_i \vee w_i)|\ge 2k-1  >k+i_0$.
			So there is   a colour $c'\in L(v_i) -\YS$. If $ |L(u_i \wedge w_i)|\ge 2$, then let $S=\{u_i,w_i\}$, we are done. Assume $|L(u_i \wedge w_i)|\le 1$.
			Hence  
			$$2+i_0 \ge |L(u_i \wedge w_i)|+ |L(v_i \wedge w_i)|+ |L(u_i \wedge v_i)| =|C_{P_i,2}|\ge 3k-|L(P_i)|\ge 3k-(2k+1).$$
			This implies that $k-3 \le i_0 \le q \le k/2$,  contrary to our assumption that $k \ge 8$.
			
			Assume next that $L(P_i)=\YS$  for   $2\le i\le i_0$.
			As each colour in $L(P_i)$ is contained in at most two lists of vertices of $P_i$, we have 
			$2(k+i_0) \ge 3k$, i.e.,  $i_0 \ge  k/2$. 
			Hence $i_0=k/2=q$ and $G=K_{5,3\star (q-1), 1\star q}$. 
			
			For each singleton part $\{v\}$ of $G$, we have $v \in \XS$ and hence $L(v) \subseteq \YS$ for each singleton part $\{v\}$. Thus
			$L(\bigcup_{i=2}^kP_i) = \YS$. 
			
			Since $C_{P_1,4}= \emptyset$,    we have $|L(P_1)|\ge 5k/3 > k+i_0 = |\YS|$. Let $c\in L(P_1)-\YS$. Then $c $ is contained in the lists of vertices in $P_1$ only, in contradiction to Observation \ref{obs-0}.  
		\end{proof}
		
		If $|\YS| = k+i_0$, then as $\Lambda_{P_i, 2}=i_0$ for $2\le i \le i_0$, we may assume that  $S'_2=\{u_2, w_2\}$ is a 2-subset of $P_2$ for which $| \bigcap_{v \in S'_2}L(v)| \ge 2$  and $\bigcap_{v \in S'_2}L(v) \cup L(P_2-S'_2) \not\subseteq \YS$.
		
		We let $\mathcal{S}'$ be the partition of $V(G)$ whose non-singleton parts are obtained from that of $\mathcal{S}$ by replacing $S_2$ with $S'_2$, i.e., $\mathcal{S}' = \{S_1, S'_2,S_3, \ldots, S_{i_0}\}$. 
		
		Instead of $\GS$, we consider the graph $G/\mathcal{S}'$. 
		We still have (\ref{eq-k+1}), i.e., 
		$$|Y_{\mathcal{S}'}| \ge  k+1, \ \ |X_{\mathcal{S}'}| \ge k+2.$$
		Then analog to the proof of Claim \ref{lem-ki0}, we can show that  
		$$|Y_{\mathcal{S}'}| \ge  k+i_0+1, \ \ |X_{\mathcal{S}'}| \ge k+i_0+2.$$
		
		Let $\mathcal{S}'' = \mathcal{S}$ if $|\YS| \ge k+i_0+1$, and $\mathcal{S}'' = \mathcal{S}'$ if $|\YS| =k+i_0$. Then 
		$$|Y_{\mathcal{S}''}| \ge  k+i_0+1, \ \ |X_{\mathcal{S}''}| \ge k+i_0+2.$$
		
		For simplicity, we assume that $\mathcal{S}''=\mathcal{S}$. Then 
		$|\XS| \ge k+i_0+2$ implies that $|\XS \cap P_i| \ge 2$ for some $i \ge i_0+1$.
		Assume $\{u,v\} \subseteq X \cap P_i$ for some $i \ge i_0+1$. Then 
		\begin{equation}
			\label{eq-2ki0}
			|\YS| \ge |L(u \vee v)| = 2k-|L(u \wedge v)| \ge 2k-i_0.
		\end{equation}
		
		Since   $\XS$ contains at most one vertex of any 2-part, we have
		\begin{equation*}
			\label{eq-6}
			|\XS|\le k+2q+1-i_0. 
		\end{equation*}

		If for some $i \ge i_0+1$, $ P_i=\{u_i,v_i, w_i\} \subseteq \XS$, then 
		\begin{equation*}
			\begin{aligned}
				|\YS| &\ge |L(P_i)| = |L(u_i)|+|L(v_i)|+|L(w_i)| \\
				& - (|L(u_i \wedge v_i)| + |L(u_i) \cap L(w_i)| + |L(v_i) \cap L(w_i)|) \\
				& \ge 3k-3i_0.
			\end{aligned}
		\end{equation*} 
		Hence $k+2q+1-i_0 \ge |\XS| \ge 3k-3i_0+1$, which implies that $k \le q+i_0\le 2q-1$, in contrary to $k\ge 2q$. 
		
		Thus $|\XS \cap P'_i| \le 2$ for $i \ge i_0+1$. This  implies that $|\XS| \le k+q+1$.
		
		On the other hand, $|\YS| \ge 2k-i_0$ (by (\ref{eq-2ki0}))	implies that  $|\XS| \ge 2k-i_0+1$. Hence   $ k+q+1 \ge |\XS| \ge 2k-i_0+1$, which implies that  $k \le i_0+q \le 2q-1$,   in contrary to $k\ge 2q$. 
		
		This completes the proof that  $G \notin \mathcal{G}_1$.

  \subsection{$G \notin \mathcal{G}_2$}
	
	Assume to the contrary that $G \in \mathcal{G}_2$.

\begin{claim}
	\label{clm-000}
Assume $P$ is a 4-part of $G$ and  $\Lambda_{P,3} \le 1$. Then $\Lambda_{P, 2} \ge 2$. If $|\Lambda_{P,2}|\ge 3$, then for any 2-subset $S$ of $P$  with 
$| \bigcap_{v \in S}L(v)| = \Lambda_{P, 2}$,   for any $x\in P-S$, $$|\bigcap_{v\in S}L(v) \cup L(x)|\ge k+2.$$  If $\Lambda_{P, 2} = 2$, then there exists a 2-subset $S$ of $P$   such that 
$| \bigcap_{v \in S}L(v) \cap C_{P,2}| = 2$, and hence  for any $x\in P-S$, $|\bigcap_{v\in S}L(v) \cup L(x)|\ge k+2$.
\end{claim}
\begin{proof}	
	Assume $P$ is a 4-part of $G$ and $\Lambda_{P,3} \le 1$. 
 Assume $\Lambda_{P,2}\ge 3$ and $S$ is a 2-subset of $P$ with   $| \bigcap_{v \in S}L(v)| = \Lambda_{P, 2}$.  Then for any $x\in P-S$, since $|\bigcap_{v \in S}L(v) \cap L(x)| \le \Lambda_{P,3} \le 1$, we have
$$|\bigcap_{v \in S}L(v) \cup L(x)|=|\bigcap_{v \in S}L(v)|+|L(x)|-|\bigcap_{v \in S}L(v) \cap L(x)|\ge \Lambda_{P,2}+k-1\ge k+2.$$
 
 Assume $\Lambda_{P,2} \le 2$. As $P$ has four 3-subsets, we have $|C_{P,3}| \le 4$. 
	As $\sum_{i=1}^3 i |C_{P,i}| = \sum_{v \in P}|L(v)|  \ge 4k$ and $\sum_{i=1}^3|C_{P,i}| \le |C_L| \le 2k+1$, it follows that 
	$|C_{P,2}| \ge 2k-9  \ge 7$ (as $k \ge 8$). Since $P$ has six 2-subsets,   there exists a 2-subset $S$ of $P$   such that 
	$|\bigcap_{v \in S}L(v)  \cap C_{P,2}| \ge 2$. Hence $\Lambda_{P,2} \ge 2$ and therefore $\Lambda_{P,2} = 2$. Moreover, there exists a 2-subset $S$ of $P$   such that 
$| \bigcap_{v \in S}L(v) \cap C_{P,2}| = 2$. For any $x\in P-S$, $$ |\bigcap_{v \in S}L(v) \cup L(x)| \ge |\bigcap_{v \in S}L(v)  \cap C_{P,2}| + |L(x)|\ge 2+k.$$ 
\end{proof}

 \begin{definition}
     \label{def-newconstruction}
     For $i=1,2,\ldots, i_0$, we  choose a subset $S_i$ of $P_i$ of size $2$ or $3$  as follows:
     \begin{enumerate}
				\item For $a+1 \le i \le i_0$,  $S_i$ is a 2-subset of $P_i$ with $|\bigcap_{v \in S_i}L(v)| = \Lambda_{P_i,2}$. 
    \item If $a=1$ and $\Lambda_{P_1,3} >0$, then  let $S_1$ be a 3-subset of $P_1$ with $|\bigcap_{v \in S_1}L(v)| = \Lambda_{P_1,3}$.  Otherwise, let $S_1$ be a 2-subset of $P_1$ with $|\bigcap_{v \in S_1}L(v)| =\Lambda_{P_1,2}$.
    \item Assume $a=2$. 
    \begin{itemize}
        \item[(i)] If  $\Lambda_{P_2,3} \ge 2$, then for $i=1,2$, let $S_i$ be a 3-subset of $P_i$ with $|\bigcap_{v \in S_i}L(v)| = \Lambda_{P_i,3}$.
    \item[(ii)] If $\Lambda_{P_1,3} > 0$ and 
    $\Lambda_{P_2,3} \le 1$, then let $S_1$ be a 3-subset of $P_1$ with $|\bigcap_{v \in S_1}L(v)|= \Lambda_{P_1,3}$, and  let  $S_2$ be a 2-subset of $P_2$ such that 
    \begin{itemize}
        \item[(A)] $| \bigcap_{v \in S_2}L(v)| = \Lambda_{P, 2}$,
        \item[(B)] $|\bigcap_{v\in S_2}L(v) \cup L(x)|\ge k+2$ for any $x\in P_2-S_2$,
        \item[(C)] Subject to (A) and (B), $|\LS(P'_1)  \bigcup L(P_2-S_2)|$ is maximum.
    \end{itemize} 
    \item[(iii)] If $\Lambda_{P_1,3}=0$, then for $i=1,2$, let  $S_i$ be a 2-subset of $P_i$ with  $|\bigcap_{v \in S_i}L(v)| = \Lambda_{P_i,2}$,   such that  $|\bigcap_{v\in S_i}L(v) \cup L(x)|\ge k+2$ for any $x\in P_i-S_i$ and subject to this condition, $|\LS(P'_1)\bigcup  \LS(P'_2)|$ is maximum.
    \end{itemize}
		\end{enumerate}	
 \end{definition}

The existence of the 2-subset $S$ in (ii) and (iii) has been proved in  Claim \ref{clm-000}.

			It follows from  the definition of $\mS$ that 
		for $i=1,2,\ldots, i_0$, $|\LS(z_i)| \ge i$. 

The same argument as in the previous subsection shows that 
\begin{equation}
	\label{eq-k+1-2}
	|\YS| \ge  k+1, \ \ |\XS| \ge k+2.
	\end{equation}

\begin{claim}
		\label{clm-ddd}
		If  $|P_i|=4$,	then $|\XS \cap P'_i| \le 2$.
	\end{claim}
	\begin{proof}
		Assume  $P_i=\{u_i,v_i,x_i,y_i\}$. Then $2 \le |P'_i| \le 3$. If $|P'_i| =2$, then the conclusion is trivial.

		Assume $|P'_i| =3$ and assume to the contrary of the claim that  $P'_i = \{z_i, x_i, y_i\} \subseteq \XS$, where $z_i$ is the identification of $u_i$ and $v_i$.  In this case, $\LS(z_i) = L(u_i \wedge v_i)$ and $\LS(x_i)=L(x_i)$, $\LS(y_i) = L(y_i)$. 
		
		If $\Lambda_{P_i,3}=0$, then $\LS(z_i)\cap L(x_i \vee y_i)=\emptyset$. By the choice of $S_i$, 
		$|L(x_i \wedge y_i)| \le |\LS(z_i)|$ and hence 
		$|L(x_i \vee y_i)| \ge 2k - |\LS(z_i)|$. Therefore $|\YS| \ge |\LS(z_i)| + |L(x_i \vee y_i)| \ge 2k$,  in contrary to (\ref{eq-2k}).

		If $\Lambda_{P_i,3}>0$, then by the choice of $S_i$, we know that  $i=a=2$, $\Lambda_{P_2,3}=1$ and $|S_1|=3$, $|P'_1|=2$. Therefore $|\XS| \le |V(\GS)| \le 2k-1$, and $|\YS| \le 2k-2$.

  Assume  $S_2=\{u_2,v_2\}$. By the choice of $S_2$ (see Claim \ref{clm-000}),  $|\LS(z_2)| \ge |L(x_i \wedge y_i)|$ and $|\LS(z_i) \cap L(x_i \vee y_i)| \le |\LS(z_i) \cap L(x_i)| + |\LS(z_i) \cap L(y_i)| \le 2 \Lambda_{P_i,3} =2$.
Hence  
  $|\YS| \ge |\LS(z_i)| + |L(x_i \vee y_i)|-2 \ge 2k-2$. So  $|\XS| = 2k-1$ and $|\YS|=2k-2$, and hence $ \XS= V(\GS)$.  This implies that $i_0=2$. 
  
  By Lemma \ref{lem-noncommoncolour},  $G$ has no $2$-part. Assume $P_3=\{u_3,v_3,w_3\}$. Then since $\Lambda_{P_3,2} \le 2$, and $P_3$ has three 2-subsets, we know that $|C_{P_3,2}| \le 6$. Therefore
  $$3k \le |L(u_3)|+|L(v_3)|+|L(w_3)| = 2|C_{P_3,2}|+|C_{P_3,1}| \le |C_L| + |C_{P_3,2}| \le 2k+6, $$
  a contradiction (as $k\ge 8$).
	\end{proof}
	
	Since $3p_3^++2p_2+p_1 \le 2k+2= 2(p_1+p_2+p_3^+)+2$   and  $G\neq K_{3\star (k/2+1), 1\star (k/2-1)}$ (i.e., $k\neq 2q-2$), we have  $$G \in \{K_{4, 3\star (q-1), 1 \star (q-1)}, K_{3\star q, 2, 1 \star (q-2)}\} \ \text{ or } k \ge 2q.$$
 	
Note that   $\XS$ contains at most one vertex of any 2-part. Combining with Claim \ref{clm-ddd}, we have
	\begin{equation*}
	\label{eq-6}
	|\XS|\le k+2q-i_0. 
	\end{equation*}

  \begin{claim}
  \label{clm-222}
     For any $i  \ge i_0+1$, $|\XS \cap P_i| \le 1$. 
  \end{claim} 
	\begin{proof}
 If $i\ge q+1$, then $P_i$ is $2^-$-part and hence $|P_i\cap \XS|\le 1$ (by Lemma \ref{lem-noncommoncolour} and (\ref{eq-2k}). 
 
 Assume $i_0+1\le i\le q$.

First we prove that $|\XS \cap P_i| \le 2$.
		Assume to the contrary that $|\XS \cap P_i| = 3$ for some $i \ge i_0+1$. Assume $ P_i=\{u_i,v_i, w_i\}$. Then 
		\begin{equation*}
		\begin{aligned}
		|\YS| &\ge |L(P_i)| = |L(u_i)|+|L(v_i)|+|L(w_i)| \\
		& - (|L(u_i \wedge v_i)| + |L(u_i) \cap L(w_i)| + |L(v_i) \cap L(w_i)|) \\
		& \ge 3k-3i_0.
  		\end{aligned}
		\end{equation*} 
		Hence $k+2q-i_0 \ge |\XS| \ge 3k-3i_0+1$, which implies that $2k+1 \le 2q+2i_0\le 4q$. As $k\ge 2q-1$, we have $k=2q-1$. Hence $q=i_0$, in contrary to $i_0+1\le i\le q$.
		
		Since $|\XS \cap P_i| \le 2$ for all $i \ge i_0+1$, we know that $|\XS| \le k+q$ (by Claim \ref{clm-ddd}).
		
		If 	   $|\XS \cap P_i| = 2$ for some $q \ge i \ge i_0+1$, then  $|\YS| \ge 2k-i_0$. Hence   $ k+q \ge |\XS| \ge 2k-i_0+1$, which implies that  $k = 2q-1$ and $i_0=q$, again in contrary to $i_0+1\le i\le q$. 
	\end{proof}
 
	It follows from  Claim \ref{clm-ddd} and Claim \ref{clm-222} that  $|\XS| \le k+i_0$ and hence $|\YS| \le k+ i_0-1$.

 Thus $|\XS \cap P'_i| \le 1$ for any $  a+1 \le i\le i_0$.  Combining with Claim \ref{clm-222}, we know that  $|\XS \cap P'_i| \le 1$ for any $i \ge a+1$.
 Since $|\XS| \ge k+2$ (by (\ref{eq-k+1-2})), it follows from Claim \ref{clm-ddd} that $a=2$ and 
 $|\XS \cap P'_{i}| =2$ for $i=1,2$, and

	\begin{equation}
	\label{eq-ys}
 |\XS| = k+2, |\YS|=k+1 \ \text{  and } 	\YS = \LS( \XS \cap P'_{1}) =\LS(\XS \cap P'_{2}).
	\end{equation}

For $i=1,2$, assume  $P_i=\{u_i,v_i,x_i,y_i\}$.

If $\Lambda_{P_2,3} \ge 2$, then $|S_2|=3$, say $S_2=\{u_2,v_2,x_2\}$. Then $|\YS|\ge |\LS(z_2)|+|L(P_2-S_2)|\ge k+2$, a contradiction.

Assume $\Lambda_{P_2,3} \le 1$. Then  (ii) or (iii) holds, and $|S_2|=2$, say $S_2=\{u_2,v_2\}$.
If $z_2 \in \XS$, say $P'_2\cap \XS=\{z_2,x_2\}$, then $|\YS|\ge |\LS(z_2)\cup L(x_2)|\ge k+2$ (by Claim \ref{clm-000}),  contrary to (\ref{eq-ys}).

Assume $z_2 \notin \XS$. Then  $x_2,y_2 \in \XS$.  Now $|L(x_2 \vee y_2)| \le |\YS| = k+1$ implies that $|L(x_2 \vee y_2)|=k+1$ and $|L(x_2 \wedge y_2)|=k-1$. 
This implies that $\Lambda_{P_2,2}=k-1$ and hence $|L(u_2 \wedge v_2)|=k-1$. 
As $k\ge 8$, i.e., $\Lambda_{P_2,2}=k-1\ge 7$, it follows from Claim \ref{clm-000} that  $| L(x_2\wedge y_2)| = \Lambda_{P, 2}\ge 2$ and $|L(x_2\wedge y_2) \cup L(v)|\ge k+2$ for any $v\in P_2-\{x_2,y_2\}$.

If (ii) holds, say $S_1=\{u_1,v_1,x_1\}$, then $L(u_2\vee v_2)=\LS(z_1)\cup L(y_1)$. This implies that $L(x_2\vee y_2)=\LS(z_1)\cup L(y_1)$, for otherwise, by 
see (ii), we should have chosen $S_2=\{x_2,y_2\}$.   So $|L(P_2)|=k+1$. Hence 
\begin{eqnarray*}
    2k-2 &=&  |L(x_2\wedge  y_2)|+| L(u_2 \wedge v_2)| \\
    &=& |L(x_2\wedge  y_2) \cap L(u_2 \wedge v_2)| + |L(x_2\wedge  y_2) \cup L(u_2 \wedge v_2)|  \\
    &\le&  |L(x_2\wedge  y_2) \cap L(u_2 \wedge v_2)|  + k+1.
\end{eqnarray*}  This implies that $L(x_2\wedge  y_2) \cap L(u_2 \wedge v_2) \ne \emptyset$, in contrary to Lemma \ref{lem-noncommoncolour}. 

Assume (iii) holds, and for $i=1,2$, $P_i=\{u_i,v_i, x_i,y_i\}$ and $S_i=\{u_i,v_i\}$. 
If $z_i\in \XS$ for some $i=1,2$, then by Claim \ref{clm-000}, $|\LS(z_i)| \ge 2$ and hence $|\YS| \ge |\LS(P'_i)| \ge k+2$, contrary to (\ref{eq-ys}). 

Assume $z_1, z_2 \notin \XS$. Then again by the choice of $S_2$, we have $L(u_2 \vee v_2) = L(x_1 \vee y_1) = L(x_2 \vee y_2)$,  $|L(x_2 \wedge y_2)| =|L(u_2 \wedge v_2)| = k-1$, and $|L(P_2) | =k+1$. This leads to the same contradiction. 
	This completes  the proof of 
	Theorem \ref{thm-main2}.

		It was proved in \cite{SHZWZ2008} that   $K_{6, 2 \star (k-3), 1 \star 2}$ is
		$k$-choosable. Combining  with Theorem \ref{thm-main2}, we conclude that 
		
		\begin{equation}
			\label{eq-v}
			p_1 \ge 3, \  \ p^+_3 \le p_1-1,  \ \  3p_3^++2p_2+p_1 \le |V|-3. 
		\end{equation}

		\section{Pseudo-$L$-colouring}

		 As described in Section \ref{sec-intro}, our strategy for proving Theorem \ref{thm-main} is to partition  $V(G)$ into a family $\mathcal{S}$ of independent sets, so that either there is a matching $\MS$ in the bipartite graph $\BS$ that covers $V(\GS)$ and hence produce an $L$-colouring of $G$, or using Hall's Theorem to produce a good partial $L$-colouring of $G$ that leads to an $L$-colouring of $G$ by using induction. The partition $\mathcal{S}$ is obtained by constructing a proper colouring $f$ of $G$, and the parts in $\mathcal{S}$ are the colour classes of $f$. For this strategy to succeed, the colouring $f$ needs to have some nice property. In this section, we define the concept of pseudo-$L$-colouring of $G$, and study properties of the partition $\mathcal{S}$ of $V(G)$ induced by such colourings.

		\begin{definition}
			\label{def-pseudo}
			A {\em pseudo   $L$-colouring} of $G$ is a proper colouring $f$ of $G$   such that $f(v) \in C_L$ for every vertex $v$, and if $f(v)=c \notin L(v)$, then   $f^{-1}(c)=\{v\}$ is a singleton $f$-class.  
		\end{definition} 
		
		In a pseudo $L$-colouring $f$ of $G$, if    $f(v) \not\in L(v)$, then we say $v$ is {\em badly $f$-coloured} (or {\em badly coloured} if $f$ is clear from the context).

		By Observation \ref{obs-0}, if $f$ is a   pseudo-$L$-colouring  of $G$, then there is a pseudo-$L$-colouring   $g$ of $G$ such that  $g(G)=C_L$ and for every badly $g$-coloured vertex $v$ of $G$, $g(v)=f(v)$. In the following, we may assume that 
		all the pseudo-$L$-colourings $f$ satisfy $f(G)=C_L$. 
		However, when we construct a pseudo-$L$-colouring $f$ of $G$, we do not need to verify that $f(G)=C_L$ (because if $f(G) \ne C_L$, then we change it to the pseudo-$L$-colouring $g$ described above).

		\begin{definition}
			\label{def-bf}
			Assume  $f$ is a pseudo   $L$-colouring of $G$. Let 
			${\mS}_f$ be the family of   $f$-classes, which is a partition of $V(G)$, {i.e., ${\mS}_f=\{f^{-1}(c):c\in C_L\}$ where $f^{-1}(c)$ is the set of all vertices coloured by $c$ under $f$.}  We denote $G/\mathcal{S}_f$, $L_{\mathcal{S}_f}$, $B_{\mathcal{S}_f}$,  $X_{\mathcal{S}_f}$ and  $Y_{\mathcal{S}_f}$ by $G_f$,   $L_f$,
			$B_f$,  $X_f$ and  by $Y_f$, respectively.  
		\end{definition}
		 \begin{figure}[h]
		\centering	\includegraphics[width=4in]{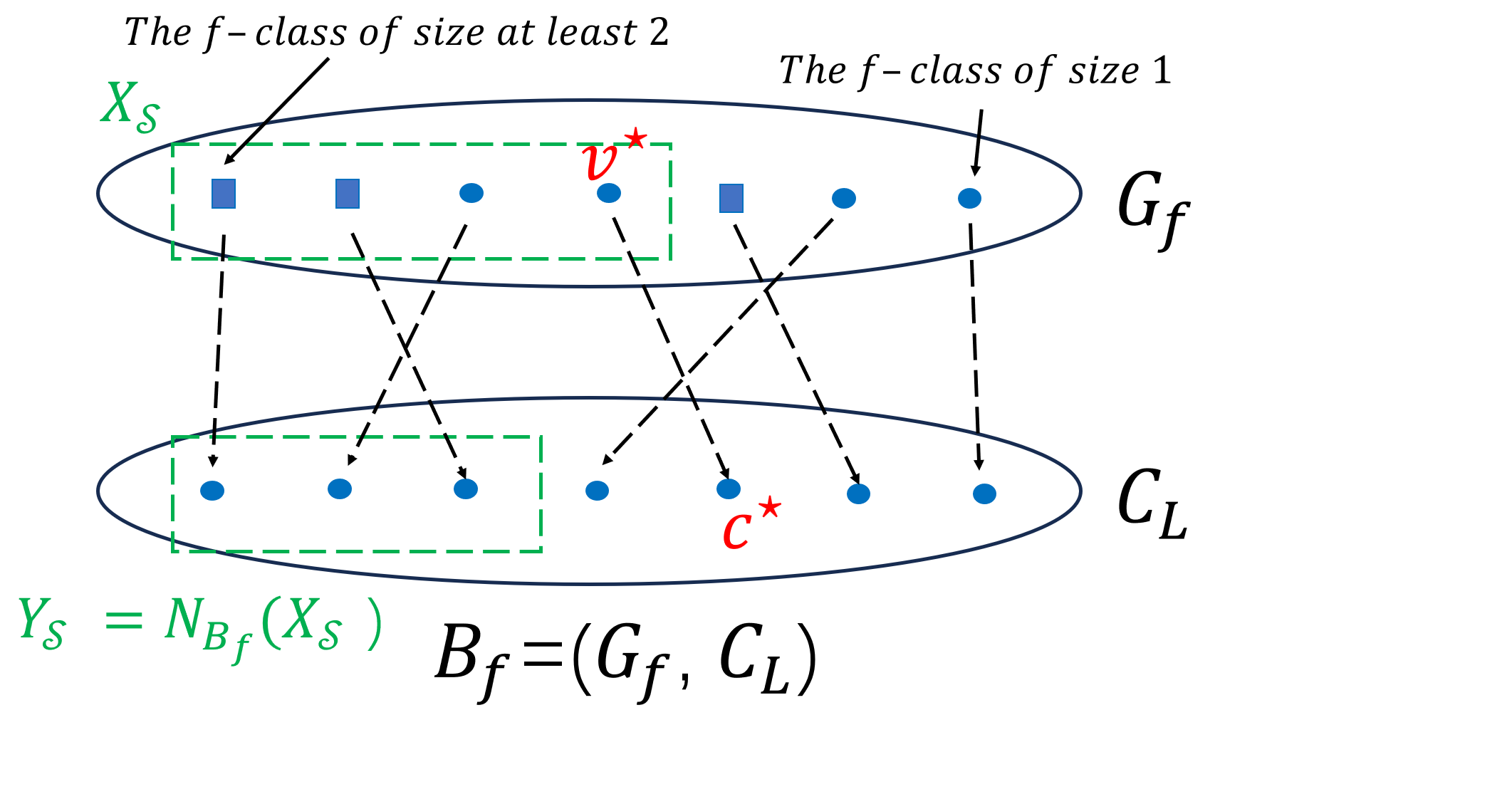}
\caption{The bipartite graph $B_f$ with partite sets $G_f$ and $C_L$. Vertices in $G_f$ are $f$-classes, some of them are singleton classes represented by solid circles, and other are $2^+$-classes, represented by solid squares. The broken arrowed line indicate the colouring $f$. The edges of $B_f$ are not drawn, and $\YS  = N_{\BS}(\XS)$. Vertex $v^*$ is contained in $\XS$ but $f(v^*)=c^* \notin \YS$. So $v^{\star}$ is a badly $f$-coloured vertex.}
\label{Figure-2}
	\end{figure}
		In the remainder of this section, assume  $f$ is a pseudo   $L$-colouring of $G$. 
  In the graph $G_f$,  $f^{-1}(c)$ is identified into a single vertex. For simplicity, we denote this vertex  by $f^{-1}(c)$. 
		So $f^{-1}(c)$ is both a subset of $V(G)$ and a vertex of $G_f$.  It will be clear from the context which one it is.
		
		Since $|X_f| > |Y_f|$, there is a colour class $f^{-1}(c) \in X_f$ such that $c \notin Y_f$. Hence $f^{-1}(c)$ is a singleton $f$-class $\{v\}$ and $v$ is badly coloured by $f$.
		
		For a subset $Q$ of $V(G_f)$, let $V(Q)$ be the subset of $V(G)$ defined as 
		$$V(Q) = \bigcup_{f^{-1}(c) \in Q} f^{-1}(c).$$
		
		Let $\ell$ be the number of $f$-classes $f^{-1}(c)$  of size $|f^{-1}(c)| \ge  2$. As $|V| > |C_L|$, $\ell \ge 1$. On the other hand,  $\lambda = |V|-|C_L| \ge \ell $ and equality holds if and only if $f(G)=C_L$ and each $f$-class has size at most 2.

		Recall that  there is a matching $M_{\mathcal{S}_f}$ in $B_f - (X_f \cup Y_f)$   that covers all vertices in $V(G_f) - X_f$.
		The matching $M_{\mathcal{S}_f}$ defines a partial $L$-colouring of $G[ \bigcup_{f^{-1}(c)\notin X_f}f^{-1}(c)]$ that colours vertices in $f^{-1}(c)$ with $c'$, where $\{c', f^{-1}(c)\} $ is an edge in $M_{\mathcal{S}_f}$. We denote this partial $L$-colouring of $G$ by $\psi_f$. The matching $M_{\mathcal{S}_f}$ maybe  not unique. In this case, we let $M_{\mathcal{S}_f}$ be an arbitrary matching that covers  $V(G_f) - X_f$.
		
		We extend $\psi_f$ to a partial $L$-colouring $\phi_f$ of $G$ by colouring each $f$-classes $f^{-1}(c) \in X_f$ of size at least 2 by colour $c$.  By definition of pseudo-$L$-colouring, for such an $f$-class $f^{-1}(c)$, $c \in \LS(f^{-1}(c))$. So $\phi_f$ is  a proper $L$-colouring of $G$.
		Denote by $X$ the set of vertices of $G$ coloured by $\phi_f$. Note that only those $f$-classes $f^{-1}(c)$ of size at least 2 contained in $X_f$ are coloured by colours from $Y_f$. So $$|\phi_f(X) \cap Y_f| \le \ell.$$
		
		If $G-X$ has an $L^{\phi_f}$-colouring $\theta$, then $\phi_f \cup \theta$ would be an $L$-colouring of $G$. Thus $G-X$ is not $L^{\phi_f}$-colourable.

		\begin{lemma}
			\label{lem-added1}
			$V_f-X_f$ contains at most $ \lambda-1 $ singletons of $G$. Moreover, if $V_f-X_f$ contains  $ \lambda-1\ge 1 $ singletons of $G$, then $\ell=\lambda$ and the following hold:
			\begin{itemize}
				\item[(1)] All $f$-classes have size 2 or 1, and there are exactly $\ell$ $f$-classes of size 2.
				\item[(2)] All the $\ell$  $f$-classes of size 2 are contained in $X_f$.
				\item[(3)] For each non-singleton part $P$ of $G$,  
				there is a singleton $f$-class $\{v\}\in X_f$ such that $v \in P$.   
				\item[(4)] If $f$ has exactly one badly coloured vertex, then $|Y_f|\ge k+1$.
			\end{itemize}
		\end{lemma}
		\begin{proof}
			It follows from the definition of $\phi_f$ that for each vertex $v$ of $ G-X$,   $\{v\} \in X_f$ is a singleton $f$-class,  and $L(v) \subseteq Y_f$.   As $|\phi^{\phi_f}(X) \cap Y_f| \le \ell$,   
			\begin{equation*}
				|L^{\phi_f}(v)| \ge k-\ell, \forall v \in V( G-X).  \label{eqn2-3}
			\end{equation*} 
			If $G_f-X_f$ contains  $ \ell$ singletons of $G$,  
			then 
			\begin{equation*}
				\chi(G-X) \le  k- \ell \text{ and } 	|V( G-X)|\le 2k+2-2\ell-\lambda  \le 2(k-\ell)+1.   \label{eq2-2a}
			\end{equation*}
			By Noel-reed-Wu Theorem, $ G-X$ is $L^{\phi_f}$-colourable, a contradiction.   
			
			So  $G_f-X_f$ contains at most $ \ell-1$  singletons of $G$. 
			
			Assume  $G_f-X_f$ contains   $ \lambda-1 $  singletons of $G$. Since $\ell \le \lambda$, we have $\ell=\lambda$ and hence each $f$-class has size at most 2, and there are exactly $\ell$ $f$-classes of size 2,  i.e.,  (1) holds. 
			We shall prove that (2)-(4) hold. 
			
			(2): Assume  to the contrary that   there is an $f$-class of size $2$  not in $X_f$.  Then at most $\ell-1$ $f$-classes are coloured by colours from $\YS$. Hence $$|L^{\phi_f}(v)|  \ge k-(\ell-1), \forall v \in V( G-X).$$ As  	\begin{equation*}
				|V( G-X)|\le 2k+2-2\ell=2(k-\ell)+2= 2(k-\ell+1) \text{ and }  \chi(G-X)\le k-\ell+1, \label{eq2-5}
			\end{equation*} 
			$ G-X$  with   list assignment $L^{\phi_f}$ 
			satisfy the condition of Noel-Reed-Wu Theorem, and hence $ G-X$ has an $L^{\phi}$-colouring, a contradiction.
			
			(3): If  (3) does not hold, then there is a non-singletone part $P$ of $G$ such that allvertices of $P$ are colourd, i.e., 
			$P$ is a 
			one non-singleton part of $G$ are contained in $X$, and hence $\chi(G-X) \le k - \lambda = k -\ell$. 	We still have $ 
			|V( G-X)|\le 2k+2-2\ell-(\lambda-1) \le 2(k-\ell)+1$.
			Hence by   Noel-Reed-Wu Theorem,   $ G-X$ has an $L^{\phi_f}$-colouring, a contradiction. 
			


		(4): Assume $v^*$ is the only badly coloured vertex. Then
	  $\{v^*\}$ is an $f$-class of size 1 in $X_f$. This implies that $|Y_f|\ge |L(v^*)| \ge k$. 
	  Assume to the contrary that $|Y_f|= k$. This implies that  for all singleton $f$-classes $\{v\}
   \in X_f$,   $L(v)=Y_f$.
		 
		Assume $f^{-1}(c) \in X_f$ is an $f$-class of size at least 2, and $P_i$   is the part of $G$ containing $f^{-1}(c)$.  As the size of $f^{-1}(c)$ is at least 2, $P_i$ is not singleton-part and hence it follows from (3) that there  is an $f$-class $\{v\} \in X_f$ such that $v \in P_i$.
  Thus,
		$L(v) = Y_f$,  $c \in L(v)$ and we can colour $v$ with colour $c$, and colour $v^*$ with $f(v)$. The resulting colouring is a pseudo $L$-colouring  of $G$ with no badly coloured vertices, i.e., an $L$-colouring of $G$, a contradiction.
		
		This completes the proof of Lemma \ref{lem-added1}.
		\end{proof}

		\begin{lemma}
			\label{lem-added}
			Assume $\lambda \ge 2$ and  $G$ has at most $\lambda-1$ singletons. Then $G_f-X_f$  contains at most $\lambda-2$ singletons of $G$.
		\end{lemma}
		\begin{proof}
			If $G$ has at most $\lambda-2$ singletons, then the conclusion is trivial. Assume $G$ has exactly $\lambda-1$ singletons (i.e., $p_1=\lambda-1$), and assume to the contrary   that all the $\lambda-1$ singletons of $G$ are contained in $G_f-X_f$. By (3) of Lemma \ref{lem-added1}, for each of the $k-\lambda+1$ $2^+$-parts $P$ of $G$,  $X_f$ has a singleton $f$-class $\{v\}$ with $v \in P$.	By Lemma \ref{lem-added1}, we have $\ell=\lambda$. By (2) of  Lemma \ref{lem-added1}, all the $\ell$ $f$-classes  of size 2 are contained in $X_f$. 
			Thus 
			\begin{equation}
				\label{eq-vx}
				|V(X_f)|\ge 2\ell+k-\lambda+1 =  \lambda+k+1.
			\end{equation}
			
			If a 2-part $P$ of $G$ is contained in $  V(X_f)$, then   $L(P)\subseteq Y_f$. By  Lemma \ref{lem-noncommoncolour},    
			$$2k\le |L(P)|\le |Y_f|.$$
			This contradicts to the fact that $|Y_f| < |C_L| = |V|-\lambda \le 2k$. 
			
			Thus for each 2-part $P$ of $G$, $|P \cap V(G_f-X_f)|\ge 1.$ (Note that a 2-part has no common colour in the lists of its vertices, so $P$ is not an $f$-class). 
			Hence 
			\begin{equation}
				\label{eq-vf}
				|V(G_f-X_f)|\ge  \lambda-1+p_2.
			\end{equation}
			As $p_1=\lambda-1$, it follows from (\ref{eq-vx}) and (\ref{eq-vf}) that 
			$$
			2k+2 = |V| =|V(X_f)| +|V(G_f-X_f)| \ge (\lambda+k+1)+ (\lambda-1+p_2)=2\lambda+k+p_2 = 2p_1+2 +k +p_2.$$ 
			So $$p_3^++p_2+p_1 = k \ge 2p_1+p_2,$$
			which implies that $p_3^+ \ge p_1$,   in contrary to (\ref{eq-v}).

			This completes the proof of Lemma \ref{lem-added}.
		\end{proof}

		\section{Near acceptable colourings}

We have shown in the previous section that the partition $\mathcal{S}$ of $V(G)$ induced by a pseudo-$L$-colouring of $G$ has some nice properties. However, for the proof of Theorem \ref{thm-main}, one more restriction need to be added to a pseudo-$L$-colouring.
In this section, we define the concept of 
 near acceptable $L$-colouring of $G$, and prove that the partition $\mathcal{S}$ of $G$ induced by a near acceptable $L$-colouring of $G$ enables us to construct a proper $L$-colouring.

			\begin{definition}
				A colour $c$ is called {\em frequent} if one of the following holds:
				\begin{itemize}
					\item[(1)] $|L^{-1}(c)| \ge k+2$.
					\item[(2)] $|L^{-1}(c) \cap T| \ge \lambda$. 
					\item[(3)] $|T|=\lambda-1\ge 1$ and $T \subseteq L^{-1}(c)$.
				\end{itemize}   
			\end{definition}

			\begin{definition}
				A pseudo $L$-colouring $f$ of $G$ is    {\em near acceptable}   if   each badly coloured vertex is coloured by a frequent colour. 
			\end{definition} 
			
			The concept of near acceptable $L$-colouring was first used in \cite{NRW2015} for the proof of Noel-Reed-Wu Theorem. 
			For the proof of Theorem \ref{thm-main}, as $G$ has one more vertex, the  definition of frequent colours is different from that in \cite{NRW2015}. Thus the near acceptable $L$-colouring in this paper is different from the one  in \cite{NRW2015}.  The difference makes it more difficult to find a near acceptable $L$-colouring of $G$.
			Nevertheless, we shall show that analog to \cite{NRW2015}, the existence of a near acceptable $L$-colouring of $G$ implies the existence of an $L$-colouring of $G$. 
   
			\begin{lemma}
				\label{lem-near}
				$G$ has no near acceptable $L$-colouring. 
			\end{lemma}
			\begin{proof}
				Assume to the contrary that $f$ is a near acceptable $L$-colouring of $G$.
				Since $|X_f| > |Y_f|$,   there is a colour class $f^{-1}(c^*)\in X_f$ with $c^*\notin Y_f$. 
				Hence $f^{-1}(c^*)=\{v^*\}$ is a badly coloured singleton $f$-class.   
				
				Since $f^{-1}(c^*)=\{v^*\}\in X_f$, we have $L(v^*) \subseteq Y_f$, and hence  $$k\le |L(v^*)|\le|Y_f|<|X_f|.$$ 
				On the other hand,  $c^*\notin Y_f$ implies that for each $f^{-1}(c) \in X_f$, there exists $v \in f^{-1}(c)$, such that $c^* \notin L(v)$.    Thus $$|L^{-1}(c^*)| \le 2k+2-|X_f| \le k+1.$$
				So $c^*$ is not a frequent colour of Type (1).

				By Lemma \ref{lem-added1}, $V_f-X_f$ contains at most $ \lambda-1 $ singletons of $G$. Hence  $$|L^{-1}(c^*)\cap T|\le \lambda-1.$$
				So $c^*$ is not a frequent colour of Type (2).
				
				If $|T|=\lambda-1\ge 1$, then by Lemma \ref{lem-added}, $| L^{-1}(c^*) \cap T| \le |V(V_f-X_f) \cap T| \le \lambda-2$. Hence $T \not\subseteq L^{-1}(c^*)$. 
				So $c^*$ is not a frequent colour of Type (3).
				

				Therefore, $c^*$ is not frequent, a contradiction.
			\end{proof}
			
			\section{Upper bound on the number frequent colours}
			\label{sec-kfrequent}
			
			This section proves that there are at most $k-1$ frequent colours. 
   Assume to the contrary that there is a set $F$ of $k$ frequent colours. We will construct a near acceptable colouring $f$ of $G$ in the following three steps: 
   \begin{itemize}
       \item[(1)] Construct a partial $L$-colouring $f_1$ of $G$ using colours from $C_L-F$, that colours as many vertices as possible, and subject to this, the coloured vertices are distributed among the parts of $G$ as evenly as possible. Let $V_1$ be the set of vertices coloured by $f_1$.
       \item[(2)] Order the parts of $G$ as $P_1,P_2, \ldots, P_k$ so that $|P_i -V_1| \ge |P_{i+1} - V_1|$ for $i=1,2,\ldots, k-1$. Colour greedily in this order the vertices of $P_i-V_1$ by a common permissible colour from $F$, until this process cannot be carried out any more. This partial $L$-colouring will be denoted by $f_2$.
       Let $V_2$ be the set of vertices coloured by $f_2$.
       \item[(3)] Extend $f_1\cup f_2$ to a near acceptable $L$-colouring (for example, if for each remaining part $P_i$, $P_i -V_1$ contains at most one vertex, then we arbitrarily colour that vertex by a remaining colour from $F$ to  obtain a near acceptable $L$-colouring of $G$). 
   \end{itemize} 	 
   The difficult part is to prove that $f_1 \cup f_2$ can be extended to a near acceptable $L$-colouring. What we really proved is that if this cannot be done, then every part of $G$ is  a $3^-$-part, which is in contrary to Theorem \ref{thm-main2}.

			In the proof, we often need to modify a partial $L$-colouring. 	
			
   \begin{definition}
       \label{def-modify}
       Assume $f$ is a partial $L$-colourings  of $G$.   For distinct colours $c_1, c_2, \ldots, c_t \in C_L$, and   distinct indices $i_1, i_2, \ldots, i_t \in \{1,2,\ldots, k\}$,
			we denote by 
			$$f(c_1 \to P_{i_1}, c_2 \to P_{i_2}, \ldots, c_t \to P_{i_t})$$ the partial $L$ colouring of $G$ obtained from $f$ by  the following operation: 
			\begin{itemize}
				\item First, for $j=1,2,\ldots, t$, 
				uncolour vertices in $f^{-1}(c_j)$ (it is allowed that $f^{-1}(c_j)=\emptyset$, i.e., $c_j$ is not used by $f$).
				\item 	Second, for $j=1,2,\ldots, t$, colour vertices in $L^{-1}(c_j) \cap  P_{i_j} $ by colour $c_j$.
			\end{itemize}
   \end{definition}

   Now we are ready to prove the following lemma.
			

   \begin{lemma}
       \label{lem-kfrequent}
       There are at most $k-1$ frequent colours.
   \end{lemma}
		\begin{proof}	
  Assume to the contrary that there is a set $F$ of $k$ frequent colours.
   A {\em valid} partial $L$-colouring $f$ of $G$ is a partial $L$-colouring of $G$ using colours from $C_L-F$.

   For a valid partial 
			$L$-colouring  $f$ of $G$, for $i=1,2,\ldots, k$,
			let 
			$$S_{f,i} = P_i \cap f^{-1}(C_L-F)$$
			be the set of coloured vertices in $P_i$.
			Let
			
			\begin{eqnarray*}
				\tau_1(f) &=&  \sum_{i=1}^k|S_{f,i}|,  \\
				\tau_2(f) &=& \sum_{i=1}^k |S_{f_1,i}|^2.
			\end{eqnarray*}

			We choose  a valid partial $L$-colouring $f_1$ of $G$  such that $$\tau(f_1) = (\tau_1(f_1),  -\tau_2(f_1))$$ is lexicographically maximum, i.e., the number of coloured vertices $\tau(f_1)$ is maximum, and subject to this, $\tau_2(f) = \sum_{i=1}^k |S_{f_1,i}|^2$ is minimum, which means that the coloured vertices are distributed among the parts of $G$ as evenly as possible.
			%
		%

		
		Let $V_1 = f_1^{-1}(C_L-F) = \bigcup_{i=1}^k S_{f,i}$ be the set of vertices coloured by $f_1$. By the maximality of $\tau_1(f_1)$, $V_1$ must have used all the colours in $C_L-F$, and hence $|C_L-F| \le |V_1|$. 
		
		If $|V-V_1| \le k$, then let $g:V-V_1\to F$ be an arbitrary injective mapping. The union $f_1\cup g$ is a near acceptable $L$-colouring of $G$, and we are done. 
		Thus   we may assume that
		\begin{equation}
			\label{eq-v-v1}
			|V-V_1| \ge k+1, \text{ and hence } |V_1| \le k+1.
		\end{equation}
		
		For $i=1,2,\ldots, k$,  
		let 
		$$ R_{f_1,i} = P_i-S_{f,i}.$$ 
  For   a colour $c \in C_L$, let   $$R_i(c) = |L^{-1}(c) \cap  R_{f_1,i}|$$ be the number of vertices in $R_{f_1,i}$ that can be coloured by $c$, and 
  $$R_i(C_L-F) = \sum_{c \in C_L-F}R_i(c)$$
  be the total number of vertices in $R_{f_1,i}$ that can be coloured by colours from $C_L-F$.

 Take Figure \ref{Figure-3} as an example. If $P_{i}=\{x_i,y_i,z_i,u_i,v_i\}$, hence $S_{f_1,i}=\{x_i,y_i\}$ and $R_{f_1,i}=\{z_i,u_i,v_i\}$. Let $L$ be a $4$-list assignment of $P_i$ with   $L(x_i)=\{1,2,5,6\}$, $L(y_i)=\{1,2,7,8\}$, $L(z_i)=\{3,4,5,6\}$, $L(u_i)=\{3,4,7,8\}$, $L(v_i)=\{5,6,7,8\}$ and $C_L=\{1,2,3,4,5,6,7,8\}$, $F=\{1,2,3,4\}$. Then $R_i(1)=R_i(2)=0$, $R_i(3)=R_i(4)=R_i(5)=R_i(6)=R_i(7)=R_i(8)=2$ and $R_i(C_L-F)=R_i(5)+R_i(6)+R_i(7)+R_i(8)=8$.

		If $c \in C_L-F$, then  $$R_i(c) \le  |f_1^{-1}(c)|,$$
  for otherwise,  	$f_1(c \to P_i)$ is a valid partial $L$-colouring of $G$ which colours more vertices than $f_1$, in contrary to the choice of $f_1$.
		
		\begin{definition}
			A colour $c \in C_L-F$ is said to be {\em movable to $P_i$} if   $R_i(c) = |f_1^{-1}(c)|$. 
		\end{definition}
		
		\begin{observation}
			\label{obs-1}
			The following facts will be used frequently in the argument below.
			\begin{enumerate}
				\item 	If $c \in C_L-F$ is movable to $P_i$, then 
				$f_1(c \to P_i)$ is a valid partial $L$-colouring of $G$ with $\tau_1(f_1(c \to P_i) ) = \tau_1(f_1)$. 
				\item  $R_i(C_L-F) \le |V_1-P_i|$, and if $R_i(C_L-F) = |V_1-P_i|$, then 
				\begin{center}
					every colour $c \in C_L-F$ with $f_1^{-1}(c) \cap P_i = \emptyset$ is movable to $P_i$.   (P1) 
				\end{center} 
				\item If $f_1^{-1}(c)$ is a singleton $f_1$-class, then $c$ is movable to $P_i$ if and only if $c \in L(R_{f,i})$.  
				\item For any choices of distinct colours  $c_1,c_2, \ldots, c_t \in C_L$ and   indices $i_1, i_2, \ldots, i_t$, $f_1(c_1 \to P_{i_1}, c_2 \to P_{i_2}, \ldots, c_t \to P_{i_t})$ is a partial $L$-colouring of $G$. 
			\end{enumerate}	 
		\end{observation}
		\begin{proof}
			(1),(3), (4) are trivial.
			
			(2): If $c\in C_L-F$ and $f^{-1}(c) \cap P_i \ne \emptyset$, then $R_i(c) = 0$, for otherwise, we can colour vertices in 
			$\{v \in R_{f_1, i}: c \in L(v)\}$ with colour $c$. By the fact that $R_i(c) \le |f^{-1}_1(c)|$, we have 
			$R_i(c) \le |f_1^{-1}(c) - P_i|$ for any colour $c \in C_L-F$. Hence $R_i(C_L-F) = \sum_{c \in C_L-F}R_i(c) \le |V_1-P_i|$, and equality holds only if $R_i(c) = |f_1^{-1}(c)|$ for all $c \in C_L-F$ with $f_1^{-1}(c) \cap P_i = \emptyset$.
		\end{proof}
		
		
		\begin{claim}
			\label{part with size 2}
			If $|P_i| = 2$, then $S_{f_1,i} \neq \emptyset$.
		\end{claim}
		\begin{proof}
			Assume to the contrary that $P_i=\{u,v\}$  and  $S_{f_1,i}= \emptyset$.
			By Lemma \ref{lem-noncommoncolour}, 
			$L(u \wedge v) = \emptyset$. Hence  $|C_L|\ge 2k$ and  $|V_1|\ge |C_L-F|\ge k$. So there are at least $k$ $f_1$-classes. As $|V_1| \le k+1$ (see (\ref{eq-v-v1})), each $f_1$-class is a singleton, except at most one $f_1$-class is of size 2.  
			
			Since $S_{f_1,i}= \emptyset$, there is an index $j_0$ such that $|f_1(S_{f_1,j_0}))| \ge 2$. Assume $c_1,c_2 \in f_1(S_{f_1,j_0})$.
			At least one of  $ f_1^{-1}(c_1), f_1^{-1}(c_2) $ is a singleton $f_1$-class. 
			
			If  
			$|C_L|=2k$, then $L(u \vee v)=C_L$ and by (3)  of Observation \ref{obs-1}, one of $c_1,c_2$, say $c_1$,   is movable to $P_i$ and $f_1^{-1}(c_1)$ is a singleton $f_1$-class. 
			If $|C_L|=2k+1$, then  there are $k+1$ $f_1$-classes, and hence each $f_1$-class is a singleton. So both $ f_1^{-1}(c_1), f_1^{-1}(c_2) $ are singleton $f_1$-classes, and at least  one of $c_1, c_2$ belongs to $L(R_{f_1,i})$ and hence is movable to $P_i$.   
			
			Assume $ f_1^{-1}(c_1)$ is a singleton $f_1$-class and $c_1$ is movable to $P_i$.  
			
			Then   $\tau_1(f_1(c_1 \to P_i))=\tau_1(f_1)$,   $\tau_2(f_1(c_1 \to P_i)) < \tau_2(f_1)$. This is in contrary to our choice of $f_1$.  
		\end{proof}

		By a re-ordering, if needed, we assume that 
		\begin{equation}
			|R_{f_1,1}| \ge |R_{f_1,2}| \ge \ldots \ge |R_{f_1,k}|. \tag{R1}  \label{R1}
		\end{equation}
		
		In the second step, starting from $i=1$ to $k$, we do the following:
		If there is a colour $c\in F$ such that $c\in \bigcap_{v\in R_{f_1,i}}L(v)$ and $c$ is not used by $R_{f_1,j}$ for $j < i$, then we colour  $R_{f_1,i}$ with $c$. 
		The step terminates   when such a colour does not exist.

		Assume the second step stopped at $i_0+1$, and hence $R_{f_1,1}, \ldots, R_{f_1,i_0}$ are coloured in the second step. 
		
		Note that in the ordering of $R_{f_1,1}, R_{f_1,2}, \ldots, R_{f_1,k}$, if some of the $R_{f_1,j}$'s has the same cardinality, then we can choose different ordering so that (\ref{R1}) still holds. Also with a given ordering of $R_{f_1,1}, R_{f_1,2}, \ldots, R_{f_1,k}$, when we colour all the vertices of $R_{f_1,i}$, there may be more than one choice of the colours. We assume that
		
		\begin{equation}
			\begin{aligned} &\text{ Subject to (\ref{R1}), the ordering of $R_{f_1,1}, R_{f_1,2}, \ldots, R_{f_1,k}$ and} \\ 
				& \text{	the colouring of  the $R_{f_1,i}$'s are chosen so that $i_0$ is maximum . }
			\end{aligned} \tag{R2} \label{R2}
		\end{equation}

		\medskip
		
		We denote by $f_2$ the colouring constructed in the second step, and by $V_2$ the set of vertices coloured in this step, and let  $V_3  =  V -V_1 -V_2$ be the set of uncoloured vertices after the second step.  Let $F_1$ be the frequent colours used in second step,
		and  let $F_2=F-F_1$.	So $|F_1|=i_0$ and $|F_2|=k-i_0$. 
		Note that it is possible that ${i_0}=0$ and $V_2 =\emptyset$.

		If $ |R_{f_1,{i_0}+1}| \le 1$, then $|V_3|\le k-{i_0}=|F_2|$, and $f_1\cup f_2$ can be extended to   a near acceptable $L$-colouring of $G$ by colouring $V_3$ injectively by $F_2$, and we are done.  
		
		Therefore  the following hold:
		\begin{equation}
			\begin{aligned}
				|R_{f_1,{i_0}+1}| &\ge 2, \\
				|V_2|&\ge 2{i_0}, \\
				|V_3|&\ge k-{i_0}+1, \\
				|V_1|&= |V|-|V_2|-|V_3|\le   k-{i_0}+1.
			\end{aligned}  	\label{eq3-3}
		\end{equation}

		Observe that for each colour   $c \in F_2$, 
		$$R_{i_0+1}(c) \le |R_{f_1, {i_0}+1}|-1,$$
		and for each colour    $c \in F_1$, 
		$$R_{i_0+1}(c) \le |R_{f_1,{i_0}+1}|.$$
		Hence  
		
		\begin{equation}
			\begin{aligned}
				R_{i_0+1}(C_L-F) &=\sum_{c \in C_L-F} R_{i_0+1}(c) \\
				&=\sum_{c \in C_L} R_{i_0+1}(c) - \sum_{c \in F_1 \cup F_2} R_{i_0+1}(c)\\
				& \ge k|R_{f_1,{i_0}+1}|-(|R_{f_1,{i_0}+1}|-1)(k-{i_0})-|R_{f_1,{i_0}+1}|{i_0}=k-{i_0},
			\end{aligned} 	\label{eq6-1a}
		\end{equation}
		and if the equality holds, then \begin{eqnarray}
		    \forall c \in F_2,  R_{i_0+1}(c) &=& |R_{f_1,i_0+1}|-1,  \label{eq-mi0c}  \\
 \forall c \in F_1,  R_{i_0+1}(c) &=& |R_{f_1,i_0+1}|.   \label{eq-mi0}
		\end{eqnarray}

		Combining (\ref{eq3-3}) with (\ref{eq6-1a}) and by (2) of Observation \ref{obs-1}, we have 
		\begin{equation}
			k-{i_0}+1 \ge |V_1|\ge |V_1 -P_{{i_0}+1}|\ge R_{i_0+1}(C_L-F)\ge k-{i_0}.   \label{eq5-1}
		\end{equation}
		
		So  $|V_1|=k-{i_0}$ or $|V_1|=k-{i_0}+1$.

		\medskip
		\noindent	{\bf Case 1}: $|V_1|=k-{i_0}$. 
		\medskip

		In this case, 
		\begin{equation}
			|V_{1}| =|
			V_1-P_{i_0+1}| = R_{i_0+1}(C_L-F)=k-i_0, \  S_{f_1,i_0+1} = V_1 \cap P_{i_0+1} = \emptyset, \ \text{ and } \  P_{{i_0}+1}=R_{f_1,{i_0}+1}.  \label{eq4-4} 
		\end{equation}
		So (P1) and (P2) holds for $i_0+1$.
		By Claim \ref{part with size 2}, $|P_{{i_0}+1}|=|R_{f_1,{i_0}+1}|\ge 3$. Hence $|V_2|\ge 3{i_0}$. This implies that $$k-{i_0}=|V_1|\le k-2{i_0}+1,$$
		and hence   ${i_0}\le 1$.

		\medskip
		\noindent		{\bf Case 1.1}: $i_0=1$. 
		\medskip

		In this case, $|V_1|=k-1$, $|V_2|\ge 3$ and $|V_3|\ge k$ (by (\ref{eq3-3}), i.e. $|V_3|\ge k-i_0+1=k$).  Since $|V|=2k+2$, we conclude that  
		$|V_2|=|R_{f_1,1}|=3$ and $|V_3|=k$.

		By (\ref{eq4-4}),   $R_2(C_L-F)= |V_1|= k-1$. This  implies that 
		$$ \sum_{c \in F}R_2(c) = \sum_{c \in C}R_2(c) -\sum_{c \in C-F}R_2(c) =3k-R_2(C_L-F) = 2k+1.$$
		Hence there is a colour $c_1\in F$ such that $R_2(C_L-F)(c_1) =|L^{-1}(c_1)\cap R_{f_1,2}|\ge 3= |R_{f_1,1}|\ge |R_{f_1,2}|$. So
		$c_1\in \bigcap_{v\in R_{f_1,2}}L(v)$.  On the other hand, by (\ref{eq4-4}),  $P_2 = R_{f_1,2}$, and by Lemma \ref{lem-noncommoncolour}, $  \bigcap_{v\in P_2}L(v) = \emptyset$, a contradiction.

		\medskip
		\noindent	{\bf Case 1.2}: $i_0=0$. 
		\medskip

		In this case, 
		\begin{equation}
			|V_1| = R_1(C_L-F) =k, \ \ |V_2|=0,\ \  |V_3|=k+2.   \label{eq4-5}
		\end{equation}
		Combining with $i_0=0$ and (P2),     for each colour $c \in F$, $R_1(c)=|R_{f_1,1}|-1$.
		\begin{claim}
			\label{p1}
			$|P_1|=|R_{f_1,1}|=3$ and $R_1(c) = 2$ for any colour $c \in F$.
		\end{claim}
		\begin{proof}
			If $|R_{f_1,1}| \ge 4$, then for any colour $c \in F$,
			$f_1(c \to P_1)$ can be extended to a near acceptable $L$-colouring of $G$ by colouring the remaining $k-1$  vertices of $V_3$  injectively with
			the remaining $k-1$ colours of $F$ (note that $|L^{-1}(c) \cap P_1|=|R_{f_1,1}|-1 \ge 3$).    
			
			Thus $|P_1|=|R_{f_1,1}|=3$ (cf. (\ref{eq4-4})). This implies that $R_1(c) = 2$ for any colour $c \in F$.
		\end{proof}	
		If there is a colour $c \in F$ such that 
		$R_2(c)\ge 2$, then 
		$f_1$ can be extended to a near acceptable $L$-colouring of $G$ by colouring a 2-subset $U_1$ of   of $R_{f_1,2}$ with a colour   $c \in \bigcap_{v \in U_1}L(v) \cap F$, colouring a 2-subset $U_2$ of $R_{f_1,1}$ by a colour from $c' \in \bigcap_{v \in U_2}L(v) \cap (F-\{c\})$,
		and colouring  the remaining $k-2$ vertices of  $V_3$  injectively with
		the remaining $k-2$ colours of $F$.    
		
		Thus 
		\begin{equation}
			\label{eq4-6}
			R_2(c) \le 1, \forall c \in F \ \text{ and }   \sum_{c \in F}R_2(c) \le k.  
		\end{equation}  
		This implies that $|R_{f_1,2}| \le 2$, for otherwise  interchanging the roles of $R_{f_1,1}$ and $R_{f_1,2}$, we would have   $R_2(c) =|R_{f_1,2}|-1\ge 2$ for all $c \in F$, in contrary to (\ref{eq4-6}).

		\begin{claim}
			\label{r}
			$|R_{f_1,i}|=1$ for $i=2,3,\ldots, k$.
		\end{claim}
		\begin{proof}
			Assume to the contrary that $|R_{f_1,2}|=2$ (as $|R_{f_1,2}|\le2$), then by Observation \ref{obs-1}, $R_2(C_L-F)\le |V_1-P_2| =|V_1|-|S_{f_1,2}| = k-|S_{f_1,2}|$
			and 
			\begin{equation}
				\label{eq4-7}
				\begin{aligned}
					\sum_{c \in F}R_2(c) &= \sum_{c \in C_L}R_2(c) - \sum_{c \in C_L-F}R_2(c) 
					= 2k - R_2(C_L-F) \ge  k + |S_{f_1,2}|.
				\end{aligned}  
			\end{equation} 
   Combining with (\ref{eq4-6}), we have $|S_{f_1,2}|=0$ and hence $R_{f_1,2}=P_2$, in contrary to Claim \ref{part with size 2}.
			By Claim \ref{part with size 2},   $|S_{f_1,2}|\ge 1$, in contrary to  (\ref{eq4-6}).  
			
			Therefore $|R_{f_1,2}|=1$ and hence $|R_{f_1,i}|=1$ for $i=2,3,\ldots, k$ (note that $|V_3|=k+2$).  
		\end{proof}
		\begin{claim}
			\label{s}
			$|S_{f_1,j}|\le 2$ for $j=2,3,\ldots, k$.
		\end{claim}
		\begin{proof}
			If $|f_1(S_{f_1,j})| \ge 2$ for some $j$, say $c_1, c_2\in f_1(S_{f_1,j})$, then 
			$\tau_1(f_1(c_1 \to P_1)) = \tau_1(f_1) $ (as (P1) holds) and  $\tau_2(f_1(c \to P_1)) <\tau_2(f_1)$,
			because    $$|S_{f_1,1}|=0  \ (\text{ by \ref{eq4-4}}),   \ |S_{f_1(c_1 \to P_1),1}|=R_1(C_L-F)(c_1),$$
			and 
			$$ |S_{f_1(c_1 \to P_1),j}|= |S_{f_1,j}|-R_1(C_L-F)(c_1) > 0,$$
			since  $|f_1(S_{f_1,j})| \ge 2$.
			This is
			in contrary to our choice of $f_1$. 
			
			Hence for each $j\in \{2,3,\ldots, k\}$, $|f_1(S_{f_1,j})| \le 1$, and $ |S_{f_1,j}|  \le  |f_1^{-1}(c_j)|$ for some $c_j \in C_L-F$. 
			As (P1) holds, $ |f_1^{-1}(c_j)| = R_1(C_L-F)(c_j)   \le 2$. So  $ |S_{f_1,j}|  \le 2$.
		\end{proof}
		Combining with Claim \ref{p1}, \ref{r} and \ref{s}, we have
		\begin{equation*}
			|R_{f_1,1}|=3, \ |S_{f_1,1}|=0, \text{ and for $2 \le j \le k$}, \
			|R_{f_1,j}| = 1, \ |S_{f_1,j}| \le 2.   \label{eq4-8}
		\end{equation*}
		So each part of $G$ is $3^-$-part,  in contrary to Theorem \ref{thm-main2}.

		\medskip
		\noindent
		{\bf Case 2: $|V_1|=k-i_0+1$.}
		\medskip

		If   $P_{i_0+1}=R_{f_1,i_0+1}$, then by Claim \ref{part with size 2},   $ |R_{f_1,i_0+1}|\ge 3$ and  $|V_2|\ge 3i_0$. By (\ref{eq3-3}),
		$|V_1|= |V|-|V_2|-|V_3| \le 2k+2 - 3i_0 - (k-i_0+1) = k-2i_0+1$, and hence 
		$i_0=0$.
		This implies that $$|V_1|=k+1, \ |V_2|=0, \  |V_3|=k+1.$$ By Observation \ref{obs-1}, $R_1(C_L-F)\le |V_1|=k+1$, we conclude that 
		$$\sum_{c \in F}R_1(c) \ge \sum_{c \in C_L}R_1(c) -\sum_{c \in C_L-F}R_1(c)  \ge 3k-R_1(C_L-F) \ge 2k-1 \ge  k+1.$$
		So there is a colour $c \in F$ such that $R_1(c)\ge2$.   We can extend $f_1$ to a near acceptable $L$-colouring of $G$ by colouring two vertices of $R_{f_1,1}$ with $c$, and the remaining $k-1$ vertices of $V_3$ injectively with the remaining $k-1$ colours of $F$. 
		
		Thus $P_{i_0+1} \ne R_{f_1, i_0+1}$, i.e., $S_{f_1, i_0+1} \ne \emptyset$.
		
		As   $S_{f_1, i_0+1} \ne \emptyset$,  $|V_1 -P_{i_0+1}|=|V_1|-|S_{f_1,i_0+1}| < |V_1|$ and   by (\ref{eq5-1}), we have  
  \begin{equation}
      \label{eq-si01}
      |V_1 -P_{i_0+1}|= k-i_0=R_{i_0+1}(C_L-F), \  |S_{f_1,i_0+1}|=1.
  \end{equation}
		So (P1) and (P2) holds for $i_0+1$.

		\begin{claim}
			\label{clm-pr0}
			For each $1 \le i \le i_0+1$, $|R_{f_1,i}| = 2$ and for $j \ge i_0+2$, $|R_{f_1,i}| \le 2$.
		\end{claim}
		\begin{proof}
			By (\ref{eq3-3}), we have $|V_2|\ge 2i_0$, $|V_3|\ge k-i_0+1$. Since $|V_1|+|V_2|+|V_3| =  2k+2$, we conclude that 
			$$|V_1|=k-i_0+1, \ |V_2|= 2i_0, \ |V_3|= k-i_0+1.$$ So  $\forall j\le i_0+1, \ |R_{f_1,j}|=2, \text{ and }  \forall j \ge i_0+2,  \  |R_{f_1,j}| \le 2.$ 
		\end{proof}

		\begin{claim}
			\label{clm-pr}
			For $1 \le i \le k$, if $|R_{f_1,i}|=2$, then $|S_{f_1,i}|=1$.
		\end{claim}
		\begin{proof}
			By Claim \ref{clm-pr0}, $|R_{f_1,1}| = \ldots = |R_{f_1, i_0+1}|$. As (P2) holds, there are $i_0$ colours $c \in F_1\subseteq F$ such that $R_{i_0+1}(c)=|R_{f_1,i_0+1}|$. Therefore,   for any index $j$ with $|R_{f_1,j}|=2$, if we re-order the parts so that $R_{f_1,j}$ and $R_{f_1, i_0+1}$ interchange positions (while the other parts stay at their position), (R1) and (R2) are satisfied. So the conclusions we have obtained  for $P_{i_0+1}$ hold for $P_j$. In particular,   for any $j$ with $|R_{f_1,j}|=2$, we have $|S_{f_1,j}|=1$.  
		\end{proof}
		
		\begin{claim}
			\label{clm-s}
			$|S_{f_1,j}| \le 2$ for all $j$.
		\end{claim}
		\begin{proof}
		{As $(P1)$ holds for $i_0+1$, $|f_1^{-1}(c)|=R_{i_0+1}(c)\le |R_{f_1,i_0+1}|=2$ for any $c\in C_L-F$.}	If $|S_{f_1,j}| \ge 3$ for some $j$, then there is a colour $c \in C_L-F$ for which the following holds:
			\begin{itemize}
				\item   $|f_1^{-1}(c) \cap P_j| =1$, or
				\item  $|S_{f_1,j}| \ge 4$, and $|f_1^{-1}(c) \cap P_j| =2$.
			\end{itemize} 
			Let $$f'_1 = f_1(c \to P_{i_0+1}).$$
			Then  $f'_1$ is a valid partial $L$-colouring of $G$ with $\tau_1(f'_1) = \tau_1(f_1)$ (as (P1) holds).
   By (\ref{eq-si01}), $|S_{f_1, i_0}|=1$. Thus either $|S_{f'_1,j}| = |S_{f_1,j}|-1 \ge 2$ and $|S_{f'_1, i_0+1}|=2$, or $|S_{f'_1,j}| = |S_{f_1,j}|-2 \ge 2$ and $|S_{f'_1, i_0+1}|=3$.
			Hence $\tau_2(f'_1 )) < \tau_2(f_1)$, in contrary to our choice of $f_1$. 
		\end{proof}
		
		It follows from Claims \ref{clm-pr} and \ref{clm-s} that 
		each part of $G$ is $3^-$-part, in contrary to Theorem \ref{thm-main2}.
		
		This completes the proof of Lemma \ref{lem-kfrequent}.
  \end{proof}

		\section{Tighter upper bound for the number of frequent colours}
		\label{sec-fewer}

		In this section and the next section, we assume that $(G,L)$ is a minimum counterexample to Theorem \ref{thm-main} with $\sum_{v \in V(G)}|L(v)|$ maximum. 

  This section proves that there are at most $k-p_1-1$ frequent colours. Assume to the contrary that there are  $k-p_1$ frequent colours. We shall construct another $k$-list assignment $L'$ of $G$ that has $k$ frequent colours. By Lemma \ref{lem-kfrequent}, $(G,L')$ is not a counterexample to Theorem \ref{thm-main}. Hence there is an $L'$-colouring $f$ of $G$. Using this colouring $f$, we construct a near-acceptable $L$-colouring of $G$, which contradicts Lemma \ref{lem-near}.

  Let $F$ be the set of frequent colours, and $F' \subseteq F $ be the set of frequent colours of Type (1). 
		
		By Lemma \ref{lem-kfrequent}, we may assume that $|F| \le k-1$. 
		If $\lambda =1$, then for any $v \in T$, all colours in $L(v)$ are frequent of Type (2), a contradiction (note that $p_1 \ge 3$, so $T \ne \emptyset$). 
		Thus  $\lambda \ge 2$. 
		
		\begin{lemma}
			\label{clm-lambdap1}
			$\lambda \le p_1+1$.
		\end{lemma}
		\begin{proof}
			For   $c\in C_L-F'$, by definition, $|L^{-1}(c)|\le k+1$. By Lemma \ref{lem-noncommoncolour},   for each $c \in F'$, $|L^{-1}(c)|\le k+p_1+2$. Therefore
			\begin{eqnarray*}
				k|V|\le \sum_{v\in V}|L(v)|=\sum_{c\in C_L}| L^{-1}(c)|\le |F'|(k+p_1+2)+|C_L-F'|(k+1).
			\end{eqnarray*}
			Hence
			\begin{equation}
				|F'|\ge \frac{k|V|-(k+1)|C_L|}{p_1+1}=\frac{k\lambda-|C_L|}{p_1+1}.   \label{eqn5-1}
			\end{equation} 
			As $|F'| < k$, we have   \begin{equation}
   \label{eq-cl}
   |C_L| >  k (\lambda-  p_1 -1).
   \end{equation} 
			Since $\lambda \ge 2$, we have $|C_L| \le 2k$. Plug this into (\ref{eq-cl}),  we have  $\lambda\le p_1+2$.
			
			If $\lambda=p_1+2$, then $|C_L|=|V| - \lambda = 2k+2 -(p_1+2) =2k-p_1\le 2k-3$ (as $p_1\ge 3$). This implies that  $G$ has no 2-part (if $\{u,v\}$ is a 2-part of $G$, then $L(u) \cap L(v) = \emptyset$ and hence $|C_L| \ge 2k$). By (\ref{eq-v}),  $2k-1 = |V|-3 \ge 3(k-p_1)+p_1$.
			Hence
			\begin{equation}
				p_1 \ge\frac{k+1}{2}.   \label{eq6-1}
			\end{equation} 
			
			By (\ref{eqn5-1}), 
			$$|F'|\ge \frac{k\lambda-|C_L|}{p_1+1}= \frac{k(p_1+2)-(2k-p_1)}{p_1+1} =\frac{(k+1)p_1}{p_1+1}=k-\frac{k-p_1}{p_1+1}> k-1.$$
			Hence  $|F'|\ge k$, a contradiction.
			Thus  $\lambda \le p_1+1$. 
		\end{proof}
		
		\begin{lemma}
			\label{clm-F}
			$F = \bigcap_{v \in T}L(v)$.
		\end{lemma}
		\begin{proof}
			If $p_1=\lambda-1$, then	each colour in $ \bigcap_{v \in T}L(v)$ is contained in at least $\lambda-1$ singleton lists, and hence is a frequent colour of Type (3).
			
			If $p_1\ge \lambda$, then each colour in $ \bigcap_{v \in T}L(v)$ is contained in at least $\lambda$ singleton lists, and hence is a frequent colour of Type (2).
			
			In any case,  $$ \bigcap_{v \in T}L(v)\subseteq F.$$
			
			On the other hand,	assume there is a frequent colour $c \notin \bigcap_{v \in T}L(v)$, say $c \notin L(v)$ for some $v \in T$, then  let $L'$ be the list assignment of $G$ defined as $L'(x) = L(x)$ for $x \ne v$ and $L'(v) =L(v) \cup \{c\}$. By our assumption that $(G,L)$ is a minimum counterexample with $\sum_{v \in V(G)}|L(v)|$ maximum,  $G$ and $L'$ is not a counterexample to Theorem \ref{thm-main}. So $G$ has an $L'$-colouring $f$. But then $f$ is a near acceptable $L$-colouring of $G$, in contrary to Lemma \ref{lem-near}. Therefore
			$F \subseteq \bigcap_{v \in T}L(v).$ 
		\end{proof}

		\begin{lemma}
			\label{lem-kfrequent}
			There are at most $k-p_1-1$  frequent colours.
		\end{lemma}
		\begin{proof}
			Assume to the contrary that  $\{c_{p_1+1}, c_{p_1+2}, \ldots, c_k\}$ is a  set of $k-p_1$ frequent colours.

			Assume $T=\{v_1, v_2, \ldots, v_{p_1}\}$. We choose $p_1$ colours $c_1, c_2, \ldots, c_{p_1}$ so that for $i=1,2,\ldots, p_1$,  
			$$c_i \in L(v_i) -\{c_{p_1+1}, \ldots, c_k\}- \{c_1, \ldots, c_{i-1}\}.$$
			As $|L(v_i)| \ge  k$,   the colour $c_i$ exists.

			Let $C'=\{c_1, c_2, \ldots, c_k\}$ and define $L'$ as follows:
			\[
			L'(v)=
			\begin{cases}
				C' &\text{if $v \in T$}, \cr
				L(v) &\text{otherwise}.
			\end{cases}	
			\]
			
			By Lemma \ref{clm-lambdap1},  $p_1\ge \lambda-1$. If $p_1\ge \lambda$, then each colour in $C'$ is Type-2 frequent with respect to $L'$. If  $p_1=\lambda-1$, then each colour in $C'$ is Type-3 frequent with respect to $L'$.  
   By  Lemma \ref{lem-kfrequent}, 
			$(G, L')$ is not a mimnimum counterexample to Theorem \ref{thm-main}. Since $C_{L'} \subseteq C_L$, we know that $(G,L')$ is not a counterexample to Theorem \ref{thm-main}. Hence $G$ has an $L'$-colouring $f$. 
			
			Note that if $v \notin T$, then $f(v) \in L(v)$. 
			We shall modify $f$ to obtain  a near acceptable $L$-colouring of $G$.
			
			Let $T'=\{v_i: 1 \le i \le p_1, c_i\in f(T)\}$.
			As   $|T-T'|=|f(T)-\{c_1, c_2, \ldots, c_{p_1}\}|$,
			there is a bijection $g: T-T'\to f(T)-\{c_1, c_2, \ldots, c_{p_1}\}$. 
			
			Let $f':V\to C_L$ be defined as follows:
			\[
			f'(v)=
			\begin{cases}
				f(v) &\text{ if $v\notin T$ }, \cr
				c_i &\text{ if $v=v_i\in T'$ }, \cr
				g(v) &\text{ if $v\in T-T'$ }.
			\end{cases}	
			\]
			Then $f'$ is a near acceptable $L$-colouring of $G$, in contradiction to Lemma \ref{lem-near}.
		\end{proof}

		\section{Final contradiction}

		We shall find a subset $X$ of $T$   and a set $F''$ of $k-p_1$ colours so that for each $c \in F''$,   
		\begin{equation*}
			|L^{-1}(c) \cap X|\ge \lambda.
		\end{equation*} 
		This would imply that all the $k-p_1$ colours in $F''$ are frequent (of Type (2)). This is in contrary to  Lemma \ref{lem-kfrequent}. 
		
		For any colour $c \in C_L-F$, $|L^{-1}(c)| \le k+1$.  Let
		$$b = \min \{k+1-|L^{-1}(c)|: c \in C_L-F\}.$$

		\begin{lemma}
			\label{lem-X}
			There is a  subset $X$ of $T$ such that 
			\begin{itemize}
				\item[(1)] $|X| \ge p_1-\lambda+1$.
				\item[(2)] $|L(X)|\le k+b$.
			\end{itemize}
			Moreover, if $b=0$ or $p_1=\lambda-1$, then $|X| \ge p_1-\lambda+2$.
		\end{lemma}
		\begin{proof}
			Let $c' \in C_L-F$ be a colour with $|L^{-1}(c')|=k+1-b$. By Lemma \ref{clm-F}, there is a vertex   $w \in T$ such that $c' \notin L(w)$. 
			Define a list assignment $L'$ as follows:
			\[
			L'(v)=
			\begin{cases}
				L(v)\cup \{c'\} &\text{$v=w$ }, \cr
				L(v) &\text{otherwise}.
			\end{cases}	
			\]
			By the maximality of $\sum_{v \in V(G)}|L(v)|$, $G$ has an   $L'$-colouring $f$. 
			We must have $f(w) = c'$  and $w$ is the only  badly coloured vertex, for otherwise $f$ is a proper $L$-colouring of $G$. 
			
			Now $f$ is a pseudo $L$-colouring of $G$. By Lemma \ref{lem-added1}, in the bipartite graph $B_f$, $V_f$ has a subset $X_f$ such that  $|X_f| > |Y_f|=|N_{B_f}(X_f)|$, and
			$V_f-X_f$ contains at most $\lambda-1$ singletons  of $G$.
			


			
			
			It is easy to see that    $w\in X_f$ and $c'\notin Y_f$.
			Let $$X = \{v \in T: \{v\} \text{ is an $f$-class in $X_f$}\}.$$
			Then $|X|=|T|-|(V_f-X_f)\cap T|\ge p_1-\lambda+1 $ and by Lemma \ref{lem-added}, if $p_1=\lambda-1$, then $|X|=|T|-|(V_f-X_f)\cap T|\ge p_1-\lambda+2 $.
			
			Since each $f$-class in $X_f$ contains a vertex $v$ for which $c' \notin L(v)$, we have
			$$|L(X)|\le |Y_f|<|X_f|\le |V|-| L^{-1}(c')| = k+1+b.$$
			So $|L(X)|\le k+b.$
			
			It remains to prove that if $b=0$, i.e., $|L^{-1}(c')|=k+1$, then $|X| \ge p_1-\lambda+2$. 
			
			Assume to the  contrary that $|L^{-1}(c')|=k+1$ and  $|X| = p_1-\lambda+1$. 
			By Lemma \ref{lem-added1}, $|Y_f|\ge k+1$ and hence $|X_f|\ge k+2$, in contrary to $|X_f|\le |V|-| L^{-1}(c')| = k+1$.

			This   completes the proof of Lemma \ref{lem-X}.
		\end{proof}

		We order the colours in $L(X)$ as $c_1,c_2, \ldots, c_t$, so that $$|L^{-1}(c_1) \cap X| \ge |L^{-1}(c_2) \cap X| \ge \ldots \ge |L^{-1}(c_t) \cap X|,$$ 
		where $t=|L(X)|$. Let $F''=\{c_1, c_2, \ldots, c_{k-p_1}\}$. 
		
		It suffices to show that $$|L^{-1}(c_{k-p_1}) \cap X| \ge \lambda,$$
		and hence  each colour $c_i \in F''$ is a frequent of Type (2).

		Let $Z=\{c_{k-p_1},c_{k-p_1+1}, \ldots, c_t\}$.
		For each $v\in X$, 
		$|L(v) \cap Z| \ge |L(v)| - (k-p_1-1) \ge  p_1+1$. 
		Hence  
		\begin{equation}
			\label{eq-z}
			|Z||L^{-1}(c_{k-p_1}) \cap X| \ge \sum_{i=k-p_1}^t |L^{-1}(c_i) \cap X| =\sum_{v \in X} |L(v) \cap Z|   \ge |X|(p_1+1).
		\end{equation} 
		
		By Lemma  \ref{lem-X},   $$|Z|= |L(X)|-(k-p_1-1)\le  p_1+1+b.$$  
		Plugging this into (\ref{eq-z}), we have 
		%
		\begin{equation*}
			\label{eq-x}
			(p_1+1+b)|L^{-1}(c_{k-p_1}) \cap X| \ge |X|(p_1+1).
		\end{equation*}
		
		This implies that 
		\begin{equation}
			\label{eq-b}
			|L^{-1}(c_{k-p_1}) \cap X|\ge \frac{ |X|(p_1+1)}{p_1+1+b}.
		\end{equation}
		
		For each $c \in C_L-F$, $| L^{-1}(c)| \le k+1-b$ (by definition of $b$). By Lemma \ref{lem-noncommoncolour},   for $c \in F$, $| L^{-1}(c)| \le k+p_1+2$. Hence
		\begin{equation}
			\label{eq-b2}
			(2k+2)k \le \sum_{v\in V}|L(v)|=\sum_{c\in C_L}| L^{-1}(c)|\le |C_L-F|(k+1-b)+|F|(k+p_1+2).
		\end{equation}
		Plugging  $|C_L|=|V|-\lambda=2k+2-\lambda$ and $|F|\le k-p_1-1$ into (\ref{eq-b2}), we have 
		\begin{equation}
			\label{eq-b3}
			(2k+2)k\le (2k+2-\lambda-(k-p_1-1))(k+1-b)+(k-p_1-1)(k+p_1+2).
		\end{equation}  
		(Note that the coefficient of $|F|$ in the right hand side of (\ref{eq-b2}) is positive.)
		
		This implies 
		\begin{equation}
			\label{eq-b4}
			b \le \frac{(p_1+3-\lambda-k)(k+1)+(k-p_1-1)(k+p_1+2)}{k+p_1+3-\lambda}.  
		\end{equation}
		
		If $\lambda=2$, then since $p_1 \ge 3$, by plugging $|X|\ge p_1-\lambda+1$ (see Lemma \ref{lem-X}) into (\ref{eq-b}), we have  $$  |L^{-1}(c_{k-p_1}) \cap X|\ge \frac{ (p_1-\lambda+1)(p_1+1)}{p_1+1+b}\ge \frac{ (p_1-1)(p_1+1)}{p_1+1+\frac{(p_1+1)(k-p_1-1)}{k+p_1+1} }=\frac{(p_1-1)(k+p_1+1)}{2k}\ge\frac{2(k+p_1+1)}{2k} >1.$$

		Since $|L^{-1}(c_{k-p_1}) \cap X|$ is an integer,  $|L^{-1}(c_{k-p_1}) \cap X|\ge 2 =\lambda$ and we are done. 
		
		Therefore   $\lambda \ge 3$ and $|C_L| \le 2k-1$. 
		By Lemma \ref{lem-noncommoncolour},  $G$ has no 2-parts. 
		By the same reason as (\ref{eq6-1}), we have
		\begin{equation*}
			p_1 \ge\frac{k+1}{2}. 
		\end{equation*}

		
		
		Combining  (\ref{eqn5-1}) with Lemma \ref{lem-kfrequent}, together with $p_1 \ge\frac{k+1}{2}$,  we have $$\frac {k-3}2 \ge  k-p_1-1\ge |F'|\ge \frac{k\lambda-|C_L|}{p_1+1} = \frac{k\lambda -(2k+2-\lambda)}{p_1+1}=\frac{(k+1)\lambda-2k-2}{p_1+1}.$$   Hence  
		\begin{eqnarray*}
			\lambda \le   \frac{   \frac{(k-3)(p_1+1)}{2}  +2k+2}{k+1}  = \frac{p_1+1}{2}+2-\frac{2(p_1+1)}{k+1} < \frac{p_1+1}{2}+1.
		\end{eqnarray*}
		Since $\lambda$ is an integer,
		\begin{equation}
			\lambda\le \frac{p_1}{2}+1.  \label{eq-lambda}
		\end{equation}
		Therefore $$p_1 \ge 2 \lambda-2 \ge \lambda+1.$$  
		Plugging this into (\ref{eq-b4}), we have
		\begin{eqnarray*}
			b &\le& \frac{(p_1+3-\lambda-k)(k+1)+(k-p_1-1)(k+p_1+2)}{k+p_1+3-\lambda}\\
			&\le&\frac{(p_1+3-\lambda-k)(k+1)+(k-p_1-1)(k+p_1+2)}{k+4} \  \ ( \text{ as $p_1\ge \lambda+1$ })\\
			&=&\frac{(p_1+1)(k-p_1-1)+(k+1)(2-\lambda)}{k+4}\\
			&\le &\frac{\frac{k-3}{2}(p_1+1)+(k+1)(2-\lambda)}{k+4} \ \ ( \text{by (\ref{eq6-1}),  {i.e., $p_1\ge \frac{k+1}{2}$ }})\\
			&=&\frac{1}{2}(p_1+1-2\lambda)+\frac{2k+2+3\lambda-\frac{7}{2}(p_1+1)}{k+4}\\
			&\le &\frac{1}{2}(p_1+1-2\lambda)+\frac{k+1/2}{k+4} \\
			&< &\frac{1}{2}(p_1+1-2\lambda)+1.
		\end{eqnarray*}
		
		It follows from (\ref{eq-lambda}) that $p_1 \ge 2\lambda -2$.
		
		If $p_1\in \{2\lambda-2, 2\lambda-1\}$, then $b=0$. This implies that $|X|\ge p_1-\lambda+2$.
		
		It follows from (\ref{eq-b}) that 
		$$|L^{-1}(c_{k-p_1}) \cap X|\ge  \frac{ |X|(p_1+1)}{p_1+1+b}\ge \frac{ (p_1-\lambda+2)(p_1+1)}{p_1+1}\ge \lambda.$$

		If $p_1\ge 2\lambda$, then $$b\le \frac{1}{2}(p_1+1-2\lambda)+\frac{1}{2}\le \frac{1}{2}(p_1+1-2\lambda)+\frac{1}{2}(p_1+1-2\lambda)= p_1+1-2\lambda.$$
		
		Hence 	$$|L^{-1}(c_{k-p_1}) \cap X|\ge  \frac{ (p_1-\lambda+1)(p_1+1)}{p_1+1+b} \ge\frac{ (p_1-\lambda+1)(p_1+1)}{2(p_1+1-\lambda)}=\frac{p_1+1}{2}\ge \lambda.$$
		This completes the whole proof of Theorem \ref{thm-main}.

  This paper characterizes all non-$k$-choosable complete $k$-partite graphs $G$ with $2k+2$ vertices. If the number of vertices of $G$ increases, and the chromatic number remains $k$, then the choice number of $G$ may increase. It was proved in \cite{NWWZ2015} that $k$-chromatic graphs with $n \ge 2k+1$ vertices have choice number at most $\lceil \frac{n+k-1}{3} \rceil$. It would be interesting to characterize graphs for which this upper bound on the choice number is sharp.

		\section*{Acknowledgement}
  We thank the referee for a careful reading of the manuscript and for many valuable comments that improve the presentation of this paper.


\begin{thebibliography}{10}
			
			\bibitem{BH1985} B. Bollob\'{a}s and A. J. Harris. 
		\newblock List-colourings of graphs.
			\newblock  Graphs Combin., 1(2):115–127, 1985.
			
			
			
		 \bibitem{CCGH} F.
			Chang, H. Chen, J. Guo and Y. Huang.
			  \newblock  On-line choice number of complete multipartite graphs: an algorithmic approach. 
			  \newblock Electron. J. Combin. 22 (2015), no. 1, Paper 1.6, 16p.
     
			\bibitem{EOOS2002} H.  Enomoto, K. Ohba, K. Ota and J. Sakamoto.
			\newblock Choice number of some complete multi-partite graphs.
			\newblock Discrete Math., 244(2002), no.1--3, 55--66.  
			
			
			\bibitem{ERT1980}
			P. Erd\H{o}s, A. L. Rubin, and H. Taylor.
			\newblock Choosability in graphs.
			\newblock Proceedings of the {W}est {C}oast {C}onference on
			{C}ombinatorics, {G}raph {T}heory and {C}omputing ({H}umboldt {S}tate
			{U}niv., {A}rcata, {C}alif., 1979), 125--157, Congress. Numer., XXVI,
			Utilitas Math., Winnipeg, Man., 1980.
			
			\bibitem{Gal}
			F. Galvin.
			\newblock  The list chromatic index of a bipartite multigraph.
			\newblock  J. Combin. Theory
			Ser. B, 63(1995), no.1, 153--158.
			
			\bibitem{GM1998} 	S. Gravier, F. Maffray.
			\newblock  Graphs whose choice number is equal to their chromatic number.
			\newblock  J. Graph Theory,	27(1998), no.2, 87--97.
			
			
			
			\bibitem{HC1992}
		R. H\"{a}ggkvist and A. Chetwynd.
		\newblock Some upper bounds on the total and list
		chromatic numbers of multigraphs.
		\newblock  J. Graph Theory, 16(1992), no.5, 503--516.
			
				\bibitem{HWZ}
				P. Huang, T.Wong, and X. Zhu.
				\newblock  Application of polynomial method to on-line
				list colouring of graphs.
				\newblock  European J Combin 33(5) (2012), 872–883.
			
			\bibitem{HZCSZ2008} W. He, L. Zhang, D. W. Cranston, Y. Shen, and G. Zheng. 
			\newblock Choice number
		of complete multipartite graphs $K_{3 \star 3, 2 \star (k-5), 1 \star 2}$.
			\newblock Discrete Math., 308(23):5871–5877, 2008.
			
			\bibitem{JT1995}
			T. R. Jensen and B. Toft. 
			\newblock Graph coloring problems. 
			\newblock Wiley-Interscience
			Series in Discrete Mathematics and Optimization. John Wiley \& Sons Inc.,
			New York, 1995. A Wiley-Interscience Publication.
			
			
			
			\bibitem{Kierstead}
			H. A. Kierstead.
			\newblock  On the choosability of complete multipartite graphs with part size three.
			\newblock  Discrete Math., 211(2000), no.1--3, 255--259.
			
			
				\bibitem{KKLZ} S. Kim, Y. Kwon, D. Liu, and X. Zhu.
				\newblock On-line list colouring of complete multipartite graphs. 
				\newblock Electron. J. Combin. 19 (2012), no. 1, Paper 41, 13 pp.
			
			
			\bibitem{KSW2011}	A. V. Kostochka, M. Stiebitz  and D. R.Woodall.
			\newblock  Ohba’s conjecture for graphs
			with independence number five.
			\newblock  Discrete Math., 311(2011), no.12, 996--1005.
			
			
			\bibitem{KMZ2014}
			J. Kozik, P. Micek, and X. Zhu.
			\newblock Towards an on-line version of Ohba's conjecture.
			\newblock European J. Combin., 36(2014), 110--121.
			
			
  
			\bibitem{Noel2013}
			J. A. Noel.
			\newblock Choosability of graphs with bounded order: Ohba's conjecture and
			beyond.
			\newblock Master’s thesis, McGill University, Montr\'{e}al, Qu\'{e}bec, 2013.
			
			
			\bibitem{NWWZ2015} J. A. Noel, D. B. West, H. Wu  and X. Zhu.
			\newblock  Beyond Ohba’s Conjecture:
			A bound on the choice number of $k$-chromatic graphs with $n$ vertices.
			\newblock   European J. Combin., 43(2015), 295--305.
			
			\bibitem{NRW2015}
			J. A. Noel, B. A. Reed, and H. Wu.
			\newblock A proof of a conjecture of Ohba.
			\newblock J. Graph Theory, 79(2015), no.2, 86--102.
			
			\bibitem{Ohba2002} K. Ohba.
			\newblock On chromatic-choosable graphs.
			\newblock  J. Graph Theory, 40(2002), no.2, 130--135.
			
			
			
			\bibitem{Ohba2004}   K. Ohba.
			\newblock Choice number of complete multipartite graphs with part size at
			most three.
			\newblock Ars Combin., 72 (2004), 133--139.
			
			\bibitem{RS2002} B. Reed and B. Sudakov.
			\newblock List colouring of graphs with at most $(2-o(1)) \chi$	vertices.
			\newblock Proceedings of the International Congress of Mathematicians, Vol.
			III (Beijing, 2002), 587--603, Higher Ed. Press, Beijing, 2002.
			
			\bibitem{RS2005} 
			B. Reed and B. Sudakov.
			\newblock  List colouring when the chromatic number is close
			to the order of the graph.
			\newblock Combinatorica, 25(2005), no.1, 117--123.
			
			\bibitem{SHZL2009}  Y. Shen, W. He, G. Zheng, and Y. Li.
			\newblock  Ohba’s conjecture is true for graphs with
			independence number at most three.
			\newblock  Appl. Math. Lett. 22(2009), no.6, 938--942.
			
			\bibitem {SHZWZ2008}
			Y. Shen, W. He, G. Zheng, Y. Wang and L. Zhang.
			\newblock On choosability of some complete multipartite graphs and Ohba's conjecture. 
			\newblock Discrete Math., 308(2008), no.1, 136--143.
			
			
			
			
			\bibitem{Vizing76} V. G. Vizing.
			\newblock 	Coloring the vertices of a graph in prescribed colors.
			\newblock   Diskret. Analiz, 1976, no.29, Metody Diskret. Anal. v Teorii Kodov i Shem, 3--10, 101.
			
			
			\bibitem{Zeilberger} D. 
			Zeilberger, 
			\newblock The method of undetermined generalization and specialization illustrated with Fred Galvin's amazing proof of the Dinitz conjecture.
			\newblock American Mathematical Monthly. 103 (3): 233–239. 
			
			
		\end{thebibliography}
	\end{document}